\documentclass[final]{siamltex}
\usepackage{amssymb}
\usepackage{bm}
\usepackage{graphicx}
\usepackage{graphics}
\usepackage{pgfplots}
\pgfplotsset{compat=newest}
\usepackage{caption2}
\usepackage{psfrag}
\usepackage{enumitem}
\usepackage{listings}
\usepackage{cite}
\usepackage{float}
\usepackage{placeins}
\usepackage{arydshln}
\usepackage{amsmath}
\usepackage{amscd}
\usepackage{amsfonts}
\usepackage{float,verbatim}
\usepackage{latexsym}

\usepackage{color}
\usepackage{multirow}
\usepackage{booktabs}
\usepackage{lscape}
\usepackage{arydshln}
\usepackage{epstopdf}
\usepackage{cases}
\usepackage{lineno}
\newtheorem{remark}{Remark}[section]
\newtheorem{assumption}{Assumption}[section]

\title{High-order smoothness and high-order collocation approximation for Volterra integral equation with multiscale and nonlinear exponent kernel}

\author{
Wenlin Qiu\footnotemark[2]\ \footnotemark[3]
\and
Chuanwei Su\footnotemark[3]
\and
Xiangcheng Zheng\footnotemark[3]\ \footnotemark[4]
}

\begin{document}

\maketitle

\renewcommand{\thefootnote}{\fnsymbol{footnote}}

\footnotetext[2]{School of Mathematics, Yunnan Normal University,
Kunming 650500, P. R. China.}

\footnotetext[3]{School of Mathematics, Shandong University,
Jinan 250100, P. R. China.}

\footnotetext[4]{Corresponding author: xzheng@sdu.edu.cn.}

\renewcommand{\thefootnote}{\arabic{footnote}}

\begin{abstract}
	We consider a Volterra integral equation with multiscale and nonlinear exponent kernel. We propose the high-order smoothing conditions, under which the initial singularity of the solutions can be eliminated up to any prescribed order. This indicates the application of the multiscale nature of the kernel on local modification of the solutions. Then a discontinuous high-order collocation method with arbitrary polynomial degree is developed and analyzed on uniform or graded meshes based on the solution regularity. This work serves as a comprehensive extension and complement to [Zheng, Qiu and Stynes, SIAM J. Numer. Anal., to appear] in both mathematical and numerical aspects.
\end{abstract}
	
	
	
	\pagestyle{myheadings}
	\thispagestyle{plain}
	\markboth{Wenlin Qiu, Chuanwei Su, Xiangcheng Zheng}{Analysis and approximation to Volterra integral equation}

\section{Introduction}
This work considers the following Volterra integral
equation with multiscale and nonlinear exponent kernel
\begin{equation}\label{vie}
u(t)+\int_0^t
\frac{u(s)}{(t-s)^{\alpha(u(s))}}\,\mathrm ds
=f(t),
\qquad t\in[0,T],
\end{equation}
where $f$ is a prescribed forcing function satisfying $f(0)=u(0)$,
and the state-dependent exponent satisfies $0\leq\alpha(\cdot) <1$.

In most studies of Volterra integral equations in the form as (\ref{vie}), the $\alpha(\cdot)$ is set as a constant, cf. \cite{Bru04,Bru17,LiangBrunner2019,WangYi21} and the references therein. In some recent studies, the exponent in the kernel is selected as a function of $t$, see e.g., \cite{LDH,LiangStynesIMAJNA,MaStynes2026,MaStynes2024,ZW_sinum20}.
One advantage of adopting the variable exponent in the kernel is to characterize the multiscale features of the underlying mechanisms, see e.g., \cite{Qiu} for the explanation of the multiscale behavior of the variable-exponent kernel and \cite{ZheMMS24} for an application of such multiscale behavior in resolving the initial non-physical singularity of subdiffusion.

Compared with the aforementioned models, the \eqref{vie} presents a more sophisticated nonlinear exponent in the kernel.  Such dependence could reflect the state-dependent intensity of the memory efforts, and has  recently been employed in \cite{cui2025} to improve the performance of graph neural networks. However, the corresponding theoretical studies remain largely undeveloped. The main challenge is that the unknown solution appears simultaneously in the integrand and in the singular exponent of the kernel, resulting in a nonlinear coupling between the solutions and the memory operator.

Recently, the work \cite{Zheng} performs  mathematical and numerical analysis for \eqref{vie}. Specifically, the well-posedness and regularity of the solutions to (\ref{vie}) has been proved, and a second-order collocation method has been rigorously analyzed. Nevertheless, there are still important problems that remain untreated:
\begin{itemize}
\item[(i)] The work \cite{Zheng} proves that the derivatives of the solutions to (\ref{vie}) are unbounded in general but, under suitable conditions on $\alpha$, the second-order derivative becomes bounded. This result suggests the possibility of eliminating initial singularities of the solutions. However, the general case, i.e. the boundedness issue of high-order derivatives and the required conditions, remains to be investigated. In particular, such singularity elimination indicates the application of the multiscale nature of the kernel on local modification of the solutions.

\item[(ii)] A second-order collocation scheme for (\ref{vie}) has been considered in \cite{Zheng}, while the general high-order scheme and the corresponding numerical analysis remain to be explored. In particular, designing and analyzing high-order schemes for nonlinear problem (\ref{vie}) seems challenging due to the coupling  between the solutions and the memory operator.
\end{itemize}

Concerning the aforementioned issues, the main contributions of the work are summarized as follows:
\begin{itemize}

\item \textit{High-order smoothness.}
We formulate a general high-order smoothing condition that can eliminate the initial singularity of the solutions up to any prescribed order. Specifically, an inductive formulation of high-order derivatives is developed, and we combine this with the Fa\`{a} di Bruno formula and partial exponential Bell polynomials to show the boundedness of the derivatives under the smoothing conditions.

\item \textit{High-order approximation.}
We develop a discontinuous high-order collocation approximation to (\ref{vie}). By the contraction mapping principle of vector-valued systems, the nonlinear algebraic system generated by the high-order collocation scheme is shown to be well-posed and stable, and the high-order convergence is established on both uniform and graded meshes depending on the solution regularity.
\end{itemize}

The rest of the paper is organized as follows. Section~\ref{sec:preliminaries}
introduces notations and existing results.
Section~\ref{sec:regularity} proves the high-order smoothness of the solutions. Section~\ref{sec:collocation-existence} formulates the nonlinear
collocation method and proves the well-posedness of numerical solutions. A preliminary mesh-point error estimate is
established in Section~\ref{sec:preliminary-error}, based on which the convergence analysis is performed in
Section~\ref{sec:error-bounds}. Numerical examples are provided in the last section to verify the theoretical analysis.

\section{Preliminaries}\label{sec:preliminaries}

\begingroup
\makeatletter
\def\@begintheorem#1#2{\par\bgroup{\normalfont #1\ #2. }\itshape\ignorespaces}
\def\@opargbegintheorem#1#2#3{\par\bgroup{\normalfont #1\ #2\ ({\upshape #3}). }\itshape\ignorespaces}
\makeatother

Let $m\ge2$ be a positive integer,
$\mathbb N=\{1,2,\ldots\}$, and $Q$ denote a
positive constant whose value is allowed to vary from one occurrence to the
next. For $1\le p\le\infty$, the space $L^p(0,T)$ is equipped with its usual
norm $\|\cdot\|_p$. When $p=\infty$, the subscript is suppressed; the same
notation $\|\cdot\|$ is also used for the uniform norm on $C[0,T]$. For
$0<\gamma<1$, $C^\gamma[0,T]$ denotes the space of H\"older-continuous
functions of exponent $\gamma$, and $C^m[0,T]$ has its standard meaning \cite{Evans2ndEd}.

We shall also use the weighted regularity class $C^{q,\nu}(0,T]$ commonly
employed for weakly singular integral equations
\cite{BPV99,Va93}. Given a nonnegative integer $q$ and
$\nu<1$, this class consists of all $z\in C[0,T]\cap C^q(0,T]$ for which, for
$k=0,1,\ldots,q$ and $0<t\le T$,
\[
 \left|z^{(k)}(t)\right|\le Q
 \begin{cases}
  1, & k<1-\nu,\\
  1+|\ln t|, & k=1-\nu,\\
  t^{1-\nu-k}, & k>1-\nu.
 \end{cases}
\]
Here $Q$ may depend on $z$, $k$, and $\nu$, but is independent of $t$.

For later use, recall the two-parameter Mittag--Leffler function
\cite[Section~3.1]{JinBook},
\[
 E_{a,b}(z)=\sum_{k=0}^{\infty}\frac{z^k}{\Gamma(ak+b)},
 \qquad a>0,\quad b\in\mathbb R.
\]
The following weakly singular Gronwall estimate is a convenient form of
\cite[Theorem~4.2]{JinBook}.

\begin{lemma}\label{lem:gronwall}
Let $\beta>0$ and $b\ge0$. Suppose that $a\in L^1(0,T)$ is nonnegative almost
everywhere and nondecreasing. If $v\in L^1(0,T)$ is nonnegative and satisfies
\[
 v(t)\le a(t)+\frac{b}{\Gamma(\beta)}
 \int_0^t\frac{v(s)}{(t-s)^{1-\beta}}\,\mathrm ds
 \quad\text{for almost every }0<t\le T,
\]
then \(v(t)\le a(t)E_{\beta,1}(bt^\beta)
 \ \text{for almost every }t\in(0,T].\)
\end{lemma}

We also record the discrete counterpart from
\cite[Theorem~6.1.19]{Bru04}.
\begin{lemma}\label{lem:discrete-gronwall}
Let $\{\gamma_n\}$ be a nonnegative, nondecreasing sequence, and let $b>0$
and $0<\nu<1$. If the nonnegative sequence $\{z_n\}$ obeys
\[
 z_n\le\gamma_n+bN^{-(1-\nu)}
 \sum_{\ell=0}^{n-1}(n-\ell)^{-\nu}z_\ell,
 \qquad 1\le n\le N,
\]
then
\(z_n\le E_{1-\nu,1}\!\left(
 b\Gamma(1-\nu)(n/N)^{1-\nu}\right)\gamma_n,
 \ 1\le n\le N.\)
\end{lemma}

Throughout this work, we follow \cite{Zheng} to impose the following assumptions on the nonlinear exponent
function $\alpha(\cdot)$.

\begin{assumption}\label{ass:nonlinear-exponent}
The exponent function $\alpha$ is subject to the following conditions.
\begin{enumerate}[label=(\roman*),leftmargin=*,itemsep=1pt,topsep=2pt]
\item There is a constant $\alpha^*\in(0,1)$ such that
$0\le\alpha(x)\le\alpha^*$ for every $x\in\mathbb R$.
\item The function $\alpha$ belongs to $C^1(\mathbb R)$. In addition, for each
$M\ge0$, there exists a $Q_M\ge0$ so that
$|\alpha(x)|+|\alpha'(x)|\le Q_M$ and
$|\alpha'(x)-\alpha'(y)|\le Q_M|x-y|$ whenever $|x|,|y|\le M$.
\end{enumerate}
\end{assumption}

The following theorem collects the well-posedness and regularity results
established in \cite
{Zheng}.

\begin{theorem}\label{thm:zheng-results}
Suppose that Assumption~\ref{ass:nonlinear-exponent} holds, and set
$\alpha_0\mathrel{:=}\alpha(f(0))$. Then the following statements hold.
\begin{enumerate}[label=(\roman*),leftmargin=*,itemsep=2pt,topsep=2pt]
\item If $f\in C[0,T]$, then equation~\eqref{vie} has a unique solution
$u\in C[0,T]$.
\item If $f\in C^{m,\alpha_0}(0,T]$ and
$\alpha\in C^m(\mathbb R)$ for some $m\in\mathbb N$, then equation~\eqref{vie}
has a unique solution $u\in C^{m,\alpha_0}(0,T]$.
\item If $f\in C^2[0,T]$, $\alpha\in C^2(\mathbb R)$, and
$\alpha_0=\alpha'(f(0))=0$, then $u\in C^2[0,T]$ and
$\|u^{(k)}\|\le Q$ for $k=1,2$.
\end{enumerate}
\end{theorem}

\section{High-order smoothness}\label{sec:regularity}

Motivated by
Theorem~\ref{thm:zheng-results}(iii), which suggests that the initial singularity can be weakened by imposing certain conditions on the initial behavior of $\alpha$, we introduce the following general
$m$th-order smoothing condition:
\begin{equation}\label{eq:reg4}
 \alpha_0=\alpha'(f(0))=\cdots=
 \alpha^{(m-1)}(f(0))=0, \qquad m\in\mathbb N.
\end{equation}

\begingroup
\raggedbottom
\makeatletter
\def\@begintheorem#1#2{\par\bgroup{\normalfont #1\ #2. }\itshape\ignorespaces}
\def\@opargbegintheorem#1#2#3{\par\bgroup{\normalfont #1\ #2\ ({\upshape #3}). }\itshape\ignorespaces}
\makeatother
\providecommand{\R}{\mathbb{R}}
\providecommand{\dd}{\,\mathrm{d}}
\providecommand{\norm}[1]{\left\lVert #1\right\rVert}
\providecommand{\abs}[1]{\left|#1\right|}
\providecommand{\one}{\mathbf{1}}
\providecommand{\pplus}[1]{\left(#1\right)_{+}}
\providecommand{\coloneqq}{\mathrel{\mathop:}=}
\providecommand{\eqqcolon}{=\mathrel{\mathop:}}
\newenvironment{volterraremarkstar}
  {\par\addvspace{6pt}\noindent\textit{Remark.}\ }
  {\par\addvspace{4pt}}
\newcommand{\volterraqedsymbol}{\ensuremath{\square}}
\newif\ifvolterraqedplaced
\newcommand{\volterraqedhere}{%
  \global\volterraqedplacedtrue\tag*{\volterraqedsymbol}}
\newenvironment{volterraproof}
  {\global\volterraqedplacedfalse\par\noindent\textit{Proof}.\ \ignorespaces}
  {\ifvolterraqedplaced
     \global\volterraqedplacedfalse
   \else
     \unskip\nobreak\hfill\volterraqedsymbol
   \fi\par\addvspace{4pt}}

Under condition~\eqref{eq:reg4}, we intend to eliminate the initial singularity of the
solutions to problem~\eqref{vie} at any prescribed order. We first
present two auxiliary lemmas.

\begingroup
\begin{lemma}\label{lem:derivative-recursion}
Let $n$ be a positive integer. Assume that $f\in C^{n+1}[0,T]$,
$\alpha\in C^{n+1}(\mathbb R)$, and the $(n+1)$th-order smoothing condition holds, and assume in addition that
$u\in C^n[0,T]$. Then,
for $1\le k\le n+1$ and $t\in(0,T]$,
\begin{equation}\label{eq:derivative-recursion}
u^{(k)}(t)=f_k(t)+\int_0^t
\frac{u(t-s)\alpha'(u(t-s))\ln s-1}
{s^{\alpha(u(t-s))}}u^{(k)}(t-s)\,\mathrm ds.
\end{equation}
Here
\(
f_1(t)\coloneqq f'(t)-u(0)$ and $
f_{k+1}(t)\coloneqq f_k'(t)-u^{(k)}(0)-h_k(t),
\)
where
\begin{align}\label{hkt}
  h_k(t)\coloneqq\int_0^t\frac{\mathrm d}{\mathrm dt}
\left(
\frac{1-u(t-s)\alpha'(u(t-s))\ln s}
{s^{\alpha(u(t-s))}}
\right)u^{(k)}(t-s)\,\mathrm ds, \quad 1\le k\le n.
\end{align}
\end{lemma}

\begin{proof}
 By Theorem~\ref{thm:zheng-results}(ii) we have
$u\in C^{n+1}(0,T]$. Then we proceed by induction on $k$. For $k=1$,
differentiating \eqref{vie} for $t\in(0,T]$ and using
$\alpha_0=0$, we obtain
\begin{align*}
u'(t)
&=f'(t)-u(0)+\int_0^t
\frac{u(t-s)\alpha'(u(t-s))\ln s-1}
{s^{\alpha(u(t-s))}}u'(t-s)\,\mathrm ds .
\end{align*}
Thus the assertion holds for $k=1$.

Suppose that the assertion holds for $k=\ell\le n$. Differentiating the
corresponding equation and using
$\alpha_0=0$ at the upper integration limit give
\begin{align*}
u^{(\ell+1)}(t)
&=f_\ell'(t)-u^{(\ell)}(0)
+\int_0^t\frac{\mathrm d}{\mathrm dt}
\left(
\frac{u(t-s)\alpha'(u(t-s))\ln s-1}
{s^{\alpha(u(t-s))}}
\right)u^{(\ell)}(t-s)\,\mathrm ds \\
&\quad+\int_0^t
\frac{u(t-s)\alpha'(u(t-s))\ln s-1}
{s^{\alpha(u(t-s))}}u^{(\ell+1)}(t-s)\,\mathrm ds \\
&=f_{\ell+1}(t)+\int_0^t
\frac{u(t-s)\alpha'(u(t-s))\ln s-1}
{s^{\alpha(u(t-s))}}u^{(\ell+1)}(t-s)\,\mathrm ds.
\end{align*}
This completes the
induction.
\end{proof}
\endgroup

\begin{lemma}\label{lem:regularity}
Suppose the same hypotheses as in Lemma
\ref{lem:derivative-recursion}. Then
$h_k\in C^{n-k}[0,T]$ for $1\le k\le n$, where $h_k$ is defined by
\eqref{hkt}.
\end{lemma}

\begin{proof}
Set $A(t)\coloneqq\alpha(u(t))$ and $b_k(t)\coloneqq c(t)u^{(k)}(t)$, where
\[
 c(t)\coloneqq 2A'(t)-u(t)A'(t)\alpha'(u(t))\ln s
 +u(t)\alpha''(u(t))u'(t).
\]
It follows that
$h_k(t)=-\int_0^t s^{-A(t-s)}\ln s \, b_k(t-s)\dd s$.

First, for every $0\le \ell\le n$, we have
$\alpha^{(\ell)}\in C^{n+1-\ell}(\R)$. Then Taylor's formula with integral
remainder yields
\begin{align*}
 \alpha^{(\ell)}(x)
 &=\alpha^{(\ell)}(u(0))+\cdots+
   \frac{\alpha^{(n)}(u(0))}{(n-\ell)!}(x-u(0))^{n-\ell} \\
 &\quad+\frac{1}{(n-\ell)!}\int_{u(0)}^x
       (x-r)^{n-\ell}\alpha^{(n+1)}(r)\dd r =\frac{1}{(n-\ell)!}\int_{u(0)}^x
       (x-r)^{n-\ell}\alpha^{(n+1)}(r)\dd r.
\end{align*}
Let $M\coloneqq\norm{u}_{C[0,T]}$ and
$K\coloneqq\max_{\abs r\le M}\abs{\alpha^{(n+1)}(r)}<\infty$.
Then
\begin{align*}
 \abs{\alpha^{(\ell)}(u(t))}
 &\le \frac{K}{(n-\ell)!}
   \left|\int_{u(0)}^{u(t)}\abs{u(t)-r}^{n-\ell}\dd r\right|
 \le Q\abs{u(t)-u(0)}^{n-\ell+1}.
\end{align*}
Since $\abs{u(t)-u(0)}\le \norm{u'}_{C[0,T]}t$, we obtain
\begin{equation}
 \abs{\alpha^{(\ell)}(u(t))}\le Qt^{n-\ell+1}.
 \label{eq:reg1}
\end{equation}
For $\ell=n+1$, estimate \eqref{eq:reg1} also follows from
$\alpha\in C^{n+1}(\R)$.

For $q\geq1$, the Fa\`{a} di Bruno formula~\cite[Theorem~11.4]{Charalambides2002} gives
\[
 \frac{\mathrm d^q}{\mathrm dt^q}\bigl(\alpha^{(\ell)}(u(t))\bigr)
 =\sum_{i=1}^q \alpha^{(\ell+i)}(u(t))
 \mathcal{B}_{q,i}\bigl(u'(t),u''(t),\ldots,
 u^{(q-i+1)}(t)\bigr),
\]
where $\mathcal{B}_{q,i}$ denotes the $(q,i)$th partial exponential
 Bell polynomial defined by
\[
 \mathcal{B}_{q,i}(x_1,\ldots,x_{q-i+1})
 =\sum_{\substack{j_1+\cdots+j_{q-i+1}=i\\
 j_1+2j_2+\cdots+(q-i+1)j_{q-i+1}=q}}
 \frac{q!}{j_1!\cdots j_{q-i+1}!}
 \prod_{\nu=1}^{q-i+1}
 \left(\frac{x_\nu}{\nu!}\right)^{j_\nu}.
\]
Using $u\in C^n[0,T]$ and \eqref{eq:reg1}, we find that
\begin{equation}
 \left|\frac{\mathrm d^q}{\mathrm dt^q}
 \bigl(\alpha^{(\ell)}(u(t))\bigr)\right|
 \le Qt^{n+1-\ell-q},
 \label{eq:reg2}
\end{equation}
for $0\le\ell\le n+1$, $1\le q\le n$, and $\ell+q\le n+1$. For $q=0$, the estimate above follows directly from \eqref{eq:reg1}.

On the other hand, since
$s^{-A(x)}=\exp\bigl(-A(x)\ln s\bigr)$, the Fa\`{a} di Bruno formula
gives for $i\geq1$,
\begin{align*}
 \frac{\mathrm d^i}{\mathrm dx^i}\bigl(s^{-A(x)}\bigr)
 &=s^{-A(x)}\sum_{j=1}^i(-\ln s)^j
 \mathcal{B}_{i,j}\bigl(A'(x),A''(x),\ldots,
 A^{(i-j+1)}(x)\bigr).
\end{align*}
Thus, the factor
multiplying $s^{-A(x)}$ is a polynomial of degree at most $i$ in
$\ln s$. Its coefficient of $(\ln s)^j$ is a finite linear combination
of products
\[
 A^{(i_1)}(x)\cdots A^{(i_j)}(x),\qquad
 i_1+\cdots+i_j=i,\qquad i_\nu\geq1.
\]
Combining this representation with
$\abs{A^{(\ell)}(x)}\le Qx^{n+1-\ell}$, $0\le\ell\le n$, gives
for $i\geq1$,
\begin{align*}
 \left|\frac{\mathrm d^i}{\mathrm dx^i}\bigl(s^{-A(x)}\bigr)\right|
 &\le Qs^{-\alpha^*}\sum_{j=1}^i
 \abs{\ln s}^j x^{j(n+1)-i}\le Qs^{-\alpha^*}\bigl(1+\abs{\ln s}^i\bigr)x^{n+1-i}.
\end{align*}
We combine this with the case $i=0$ to get
\[
 \left|\frac{\mathrm d^i}{\mathrm dx^i}\bigl(s^{-A(x)}\bigr)\right|
 \le
 \begin{cases}
 Qs^{-\alpha^*}\bigl(1+\abs{\ln s}^i\bigr)x^{n+1-i},&i\ge1,\\
 Qs^{-\alpha^*},&i=0.
 \end{cases}
\]

Similarly, by selecting the appropriate values of $\ell$ and $q$ in
\eqref{eq:reg2} and using $u\in C^n[0,T]$, we obtain
\begin{align*}
 \left|\frac{\mathrm d^r}{\mathrm dx^r}\bigl(2A'(x)\bigr)\right|
 &\le Qx^{n-r}, \qquad 0\le r\le n,\\
 \left|\frac{\mathrm d^r}{\mathrm dx^r}
 \bigl(u(x)A'(x)\alpha'(u(x))\bigr)\right|
 &\le Qx^{2n-r},  \qquad 0\le r\le n-1,\\
 \left|\frac{\mathrm d^r}{\mathrm dx^r}
 \bigl(u(x)\alpha^{(2)}(u(x))u'(x)\bigr)\right|
 &\le Qx^{n-1-r},  \qquad 0\le r\le n-1.
\end{align*}
Hence $\abs{c^{(r)}(x)}\le Qx^{n-1-r}(1+\abs{\ln s})$ for
$0\le r\le n-1$.

Since
$
 b_k^{(\ell)}(x)=\sum_{r=0}^{\ell}\binom{\ell}{r}
 c^{(r)}(x)u^{(k+\ell-r)}(x)$,
then the preceding estimate and $u\in C^n[0,T]$ imply
\[
 \abs{b_k^{(\ell)}(x)}
 \le Q(1+\abs{\ln s})x^{n-1-\ell}
 \le Q(1+\abs{\ln s})x^{n-k-\ell},
 \quad 1\le k\le n,\quad 0\le\ell\le n-k.
\]
Therefore, we have
\begin{align}
 \Big|\frac{\mathrm d^j}{\mathrm dx^j}
      & \bigl(s^{-A(x)}b_k(x)\bigr)\Big|
 \le Q\sum_{i=0}^j
 \left|\frac{\mathrm d^i}{\mathrm dx^i}s^{-A(x)}\right|
 \abs{b_k^{(j-i)}(x)} \notag\\
 &\le Q\Bigl[s^{-\alpha^*}(1+\abs{\ln s})x^{n-k-j} \notag+
 \sum_{i=1}^j s^{-\alpha^*}
 \bigl(1+\abs{\ln s}^i\bigr)(1+\abs{\ln s})
 x^{n+1-i}x^{n-k-j+i}\Bigr] \notag\\
 &\le Qs^{-\alpha^*}\bigl(1+\abs{\ln s}^{j+1}\bigr)x^{n-k-j},
 \qquad 0\le j\le n-k.
 \label{eq:reg3}
\end{align}
For $0\le j\le n-k-1$, it follows that
$\lim_{x\to0}\frac{\mathrm d^j}{\mathrm dx^j}
\bigl(s^{-A(x)}b_k(x)\bigr)=0$.
Consequently,
\[
 h_k^{(j)}(t)=-\int_0^t \ln s\,
 \frac{\mathrm d^j}{\mathrm dt^j}
 \bigl(s^{-A(t-s)}b_k(t-s)\bigr)\dd s,
 \qquad 0\le j\le n-k.
\]
Moreover, by \eqref{eq:reg3},
\[
 \left| \ln s\,
 \frac{\mathrm d^j}{\mathrm dt^j}
 \bigl(s^{-A(t-s)}b_k(t-s)\bigr)\right|
 \le Q\abs{\ln s}s^{-\alpha^*}
 \bigl(1+\abs{\ln s}^{j+1}\bigr)
 (t-s)^{n-k-j}.
\]
The right-hand side is integrable on $(0,T)$. The dominated convergence
theorem, together with
$s^{-A(x)}b_k(x)\in C^{n-k}[0,T]$, therefore gives
$h_k^{(j)}\in C[0,T]$ for $0\le j\le n-k$. Hence
$h_k\in C^{n-k}[0,T]$.
\end{proof}

\vskip 1mm
Based on the above analysis, we derive the following theorem.
\begin{theorem}\label{thm:volterra-regularity}
Suppose that $f\in C^m[0,T]$ and $\alpha\in C^m(\R)$. If $m$th-order smoothing condition \eqref{eq:reg4} is satisfied,
then $u\in C^m[0,T]$.
\end{theorem}

\begin{proof}
We proceed by induction on $m$. The case $m=1$ was established in
\cite{Zheng}. Assume that the assertion holds for $m\leq n$ and consider the
case $m=n+1$.
Let $f\in C^{n+1}[0,T]$ and $\alpha\in C^{n+1}(\R)$, and suppose that the
$(n+1)$th-order smoothing condition holds. Then $u\in C^n[0,T]$ by the induction hypothesis, whereas
Theorem~\ref{thm:zheng-results}(ii) yields $u\in C^{n+1}(0,T]$.

We first show that
$f_k\in C^{n+1-k}[0,T]$ by an inner induction. For $k=1$, the assumption
$f\in C^{n+1}[0,T]$ implies that
$f_1=f'-u(0)\in C^n[0,T]$. Suppose that
$f_k\in C^{n+1-k}[0,T]$. Then
$f_k'\in C^{n-k}[0,T]$, while Lemma~\ref{lem:regularity} gives
$h_k\in C^{n-k}[0,T]$. Consequently,
$f_{k+1}=f_k'-u^{(k)}(0)-h_k\in C^{n-k}[0,T]$.
This completes the inner induction. In particular, $f_n\in C^1[0,T]$, and
hence
$f_{n+1}=f_n'-u^{(n)}(0)-h_n\in C[0,T]$.

Now we intend to bound
$u^{(n+1)}$.
 The kernel in
\eqref{eq:derivative-recursion} satisfies
\[
\left|\frac{u(t-s)\alpha'(u(t-s))\ln s-1}
 {s^{\alpha(u(t-s))}}\right|\le Qs^{-\gamma},~~\gamma:=(\alpha^*+1)/2<1.
\]
It follows from~\eqref{eq:derivative-recursion} that
\begin{align}\label{un1}
  |u^{(n+1)}(t)|
 \le \|f_{n+1}\| +Q\int_0^t s^{-\gamma} |u^{(n+1)}(t-s)|\dd s.
\end{align}

In order to apply Lemma~\ref{lem:gronwall} to \eqref{un1}, we need to show $u^{(n+1)}(t)\in L^1(0,T)$.
Multiplying both sides of \eqref{un1} by $e^{-\lambda t}$ for some $\lambda>0$, we obtain
\[
e^{-\lambda t}|u^{(n+1)}(t)|
\leq
\|f_{n+1}\|e^{-\lambda t}
+
Q\int_0^t
s^{-\gamma}e^{-\lambda s}
e^{-\lambda(t-s)}
|u^{(n+1)}(t-s)|\,ds.
\]
Integrating over $t\in(0,T)$ gives
\[
\begin{aligned}
\int_0^T e^{-\lambda t}|u^{(n+1)}(t)|\,dt
&\leq
\|f_{n+1}\|\int_0^T e^{-\lambda t}\,dt \\
&\quad+
Q\int_0^T\int_0^t
s^{-\gamma}e^{-\lambda s}
e^{-\lambda(t-s)}
|u^{(n+1)}(t-s)|\,ds\,dt,
\end{aligned}
\]
in which the double integral can be estimated as
\[
\begin{aligned}
\int_0^T\int_0^t
& s^{-\gamma}e^{-\lambda s}
e^{-\lambda(t-s)}
|u^{(n+1)}(t-s)|\,ds\,dt \\
&\qquad=
\int_0^T s^{-\gamma}e^{-\lambda s}
\int_0^{T-s}
e^{-\lambda \theta}|u^{(n+1)}(\theta)|\,d\theta\,ds \\
&\qquad\leq
\left(
\int_0^T s^{-\gamma}e^{-\lambda s}\,ds
\right)
\left(
\int_0^T e^{-\lambda \theta}|u^{(n+1)}(\theta)|\,d\theta
\right).
\end{aligned}
\]
Since
\[
\int_0^T s^{-\gamma}e^{-\lambda s}\,ds
\leq
\int_0^\infty s^{-\gamma}e^{-\lambda s}\,ds
=
\Gamma(1-\gamma)\lambda^{\gamma-1},
\]
it follows that
\[
\int_0^T e^{-\lambda t}|u^{(n+1)}(t)|\,dt
\leq
\frac{\|f_{n+1}\|}{\lambda}
+
Q\Gamma(1-\gamma)\lambda^{\gamma-1}
\int_0^T e^{-\lambda t}|u^{(n+1)}(t)|\,dt.
\]
Choose $\lambda$ sufficiently large such that $Q\Gamma(1-\gamma)\lambda^{\gamma-1} \leq \frac12$, then
\[
   e^{-\lambda T} \int_0^T |u^{(n+1)}(t)|\,dt \leq \int_0^T e^{-\lambda t}|u^{(n+1)}(t)|\,dt \leq \frac{2\|f_{n+1}\|}{\lambda}.
\]
which yields $\int_0^T |u^{(n+1)}(t)|\,dt \leq 2\lambda^{-1}e^{\lambda T} \|f_{n+1}\| $.
 Thus, $u^{(n+1)}(t)\in L^1(0,T)$ and an application of Lemma~\ref{lem:gronwall} to \eqref{un1} gives $|u^{(n+1)}\|_{L^\infty(0,T)}\le Q \|f_{n+1}\|$.

Finally, letting $t\to0^+$ in~\eqref{eq:derivative-recursion} and using dominated convergence theorem, we obtain $\lim_{t\to0^+}u^{(n+1)}(t)=f_{n+1}(0).$ Finally, it follows from $u\in C^n[0,T]\cap C^{n+1}(0,T]$ and the mean value theorem that $u^{(n+1)}$ extends continuously to $t=0$, and hence $u\in C^{n+1}[0,T]$. This completes the induction.
\end{proof}

\section{High-order collocation scheme}
\label{sec:collocation-existence}

In this section, we  propose the high-order collocation scheme and prove its well-posedness.
Let $N$ be a positive integer and introduce the partition
$0=t_0<t_1<\cdots<t_N=T$, where $t_j\coloneqq(j\tau)^r$,
$\tau\coloneqq T^{1/r}/N$, and $r\ge1$. Set
$\tau_j\coloneqq t_j-t_{j-1}$, $j=1,2,\ldots,N$, and note that
$\tau'\coloneqq\max_{1\le j\le N}\tau_j\le Tr/N\le Q\tau$. Let
$I_h\coloneqq\{t_n:n=0,1,\ldots,N\}$ and define
\[
 S_{m-1}^{(-1)}(I_h)\coloneqq
 \left\{v:v|_{(t_n,t_{n+1}]}\in\pi_{m-1},\ 0\le n\le N-1\right\},
\]
where $\pi_{m-1}$ denotes the space of polynomials of degree at most
$m-1$. We seek an approximation $U(t)\in S_{m-1}^{(-1)}(I_h)$ to $u(t)$.
For prescribed parameters $0<c_1<c_2<\cdots<c_m\le1$, let
$t_{n,i}\coloneqq t_{n-1}+c_i\tau_n$ and set
$X_h\coloneqq\{t_{n,i}:n=1,\ldots,N,\ i=1,\ldots,m\}$. The collocation
equations read
\begin{equation}\label{vie01}
 U(t)+\int_0^t\frac{U(s)}{(t-s)^{\alpha(U(s))}}\dd s=f(t),
 \quad t\in X_h.
\end{equation}

\begin{remark}
For $s\in(0,1]$,
let $\ell_1,\ldots,\ell_m$ be the Lagrange basis functions associated
with the nodes $c_1,\ldots,c_m$, namely
\[
 \ell_j(s)\coloneqq\prod_{\substack{1\le k\le m\\k\ne j}}
 \frac{s-c_k}{c_j-c_k},\qquad 1\le j\le m.
\]
The collocation solution therefore admits the representation
\begin{equation}\label{CollMethod1}
 U(t_{n-1}+s\tau_n)
 =\sum_{j=1}^m\ell_j(s)U(t_{n,j}).
\end{equation}
For $1\le n\le N$, write
$z_n\coloneqq\max_{1\le i\le m}\abs{U(t_{n,i})}$. With
\vspace{-1mm}
\begin{align*}
  L_{m-1}\coloneqq\sup_{0\le s\le1}\sum_{j=1}^m
\abs{\ell_j(s)}<\infty,
\end{align*}
then equation~\eqref{CollMethod1} yields
\begin{equation}\label{eq:coll1}
 \abs{U(t_{n-1}+s\tau_n)}\le L_{m-1}z_n.
\end{equation}
\end{remark}

\vspace{-6mm}
\subsection{A priori bound}

We deduce an a priori bound for the numerical solution by the following lemma.

\begin{lemma}\label{lem:collocation-stability}
Let $d\in\{1,2,\ldots,N\}$ and suppose that
\[
 \tau_n\le\tau^*\coloneqq
 \left(\frac{1-\alpha^*}{2\max\{1,T^{\alpha^*}\}L_{m-1}}
 \right)^{\!1/(1-\alpha^*)},
 \qquad n=1,\ldots,d.
\]
If $U$ has been defined on $[0,t_d]$, then
\[
 \max_{0\le t\le t_d}\abs{U(t)}\le C_0\norm{f},
\]
where $C_0$ is independent of $n$ and $N$.
\end{lemma}

\begin{proof}
For $n\in\{1,\ldots,d\}$ and $1\le i\le m$, equation \eqref{vie01} gives
\begin{align*}
 U(t_{n,i})
 &=f(t_{n,i})-\int_0^{t_{n,i}}
 \frac{U(s)}{(t_{n,i}-s)^{\alpha(U(s))}}\dd s \\
 &=f(t_{n,i})-\sum_{j=1}^{n-1}\int_{t_{j-1}}^{t_j}
 \frac{U(s)}{(t_{n,i}-s)^{\alpha(U(s))}}\dd s
 -\int_{t_{n-1}}^{t_{n,i}}
 \frac{U(s)}{(t_{n,i}-s)^{\alpha(U(s))}}\dd s.
\end{align*}
When $n=1$, the sum over $j$ is understood to be zero.
Using \eqref{eq:coll1}, we obtain
\begingroup
\setlength{\jot}{0pt}
\begin{align}
 \abs{U(t_{n,i})}
 &\le\norm f+
 \sum_{j=1}^{n-1}\int_{t_{j-1}}^{t_j}
 \frac{\abs{U(s)}}{(t_{n,i}-s)^{\alpha(U(s))}}\dd s
 +\int_{t_{n-1}}^{t_{n,i}}
 \frac{\abs{U(s)}}{(t_{n,i}-s)^{\alpha(U(s))}}\dd s \notag\\
 &\le\norm f+L_{m-1}\max\{1,T^{\alpha^*}\}\Biggl[
 \sum_{j=1}^{n-1}z_j
 \int_{t_{j-1}}^{t_j}\frac{\dd s}{(t_{n,i}-s)^{\alpha^*}}\notag+z_n\int_{t_{n-1}}^{t_{n,i}}
 \frac{\dd s}{(t_{n,i}-s)^{\alpha^*}}\Biggr]\notag\\
 &\le\norm f+L_{m-1}\max\{1,T^{\alpha^*}\}\Biggl[
 \sum_{j=1}^{n-1}z_j\int_{t_{j-1}}^{t_j}
 \frac{\dd s}{(t_{n,i}-s)^{\alpha^*}}
 +\frac{\tau_n^{1-\alpha^*}}{1-\alpha^*}z_n\Biggr].
 \label{eq:coll2}
\end{align}
\endgroup
Here and below, $Q$ is independent of $n$ and $N$. For
$1\le j\le n-1$ and $s\in[t_{j-1},t_j]$,
\[
 \frac{t_n-s}{t_{n,i}-s}
 =1+\frac{(1-c_i)\tau_n}{t_{n,i}-s}
 \le1+\frac{(1-c_i)\tau_n}{t_{n,i}-t_{n-1}}
 =1+\frac{1-c_i}{c_i}=\frac1{c_i}\le\frac1{c_1}.
\]
Thus $1/(t_{n,i}-s)^{\alpha^*}\le
c_1^{-\alpha^*}/(t_n-s)^{\alpha^*}$. Substituting this inequality into
\eqref{eq:coll2} and using
$c_1^{-\alpha^*}>1$, we find
\[
 \abs{U(t_{n,i})}\le\norm f+L_{m-1}\max\{1,T^{\alpha^*}\}
 \left[\sum_{j=1}^{n-1}\frac{z_j}{c_1^{\alpha^*}}
 \int_{t_{j-1}}^{t_j}\frac{\dd s}{(t_n-s)^{\alpha^*}}
 +\frac{\tau_n^{1-\alpha^*}}{1-\alpha^*}z_n\right].
\]
This bound holds for every $1\le i\le m$ and its right-hand side is
independent of $i$. Therefore,
\[
 z_n\le\norm f+\frac12z_n+
 \frac{L_{m-1}\max\{1,T^{\alpha^*}\}}
 {(1-\alpha^*)c_1^{\alpha^*}}
 \sum_{j=1}^{n-1}z_j
 \left[(t_n-t_{j-1})^{1-\alpha^*}
 -(t_n-t_j)^{1-\alpha^*}\right].
\]
After rearrangement,
\begin{equation}
 z_n\le Q\sum_{j=1}^{n-1}
 \left[(t_n-t_{j-1})^{1-\alpha^*}
 -(t_n-t_j)^{1-\alpha^*}\right]z_j+Q\norm f.
 \label{eq:coll3}
\end{equation}
For $1\le j\le n-1$, one has
\[
 t_n-t_j=T\left(\frac nN\right)^r-T\left(\frac jN\right)^r
 \ge rT\left(\frac jN\right)^{r-1}\frac{n-j}{N},
 \quad
 t_j-t_{j-1}=\tau_j\le\frac{rTj^{r-1}}{N^r}.
\]
Hence
\begin{align*}
 &(t_n-t_{j-1})^{1-\alpha^*}-(t_n-t_j)^{1-\alpha^*}\le(1-\alpha^*)(t_j-t_{j-1})(t_n-t_j)^{-\alpha^*}\\
 &\quad\le(1-\alpha^*)\frac{rTj^{r-1}}{N^r}
 \left[rT\left(\frac jN\right)^{r-1}
 \frac{n-j}{N}\right]^{-\alpha^*}\\
 &\quad=(1-\alpha^*)(rT)^{1-\alpha^*}
 \left(\frac jN\right)^{(r-1)(1-\alpha^*)}
 \frac{N^{-(1-\alpha^*)}}{(n-j)^{\alpha^*}}
 \le\frac{QN^{-(1-\alpha^*)}}{(n-j)^{\alpha^*}}.
\end{align*}
Substitution into \eqref{eq:coll3} gives
\(
 z_n\le QN^{-(1-\alpha^*)}\sum_{j=1}^{n-1}
 (n-j)^{-\alpha^*}z_j+Q\norm f.
\)
Using Lemma \ref{lem:discrete-gronwall} now yields $z_n\le Q\norm f$,
$1\le n\le d$. Finally, \eqref{eq:coll1} gives
$\max_{0\le t\le t_d}\abs{U(t)}\le C_0\norm f$.
\end{proof}

\subsection{Existence and uniqueness}

Define $\nu^-\coloneqq(1-\alpha^*)/2$ and
$\nu^+\coloneqq(1+\alpha^*)/2$, and let
\begin{align*}
 M_0&\coloneqq\norm f+
 \frac{\max\{1,T\}C_0\norm f\,T^{1-\alpha^*}}
 {c_1^{\alpha^*}(1-\alpha^*)},~
 M_1\coloneqq M_0+C_0\norm f+1,\\
 M_2&\coloneqq
 \frac{2M_1Q_{L_{m-1}M_1}L_{m-1}
 \max\{T^{\nu^-}\abs{\ln T},(\nu^-e)^{-1}\}}{\nu^-}.
\end{align*}
Let $M_3$ be the positive root of
$\frac{L_{m-1}}{1-\alpha^*}x^2+\frac{M_2}{2}x=\frac12$; that is,
\vspace{-1mm}
\[
 M_3=\frac{1-\alpha^*}{4L_{m-1}}
 \left(\sqrt{M_2^2+\frac{8L_{m-1}}{1-\alpha^*}}-M_2\right).
\]
Finally, set
\[
 \tau_{\max}\coloneqq\min\left\{
 M_3^{1/\nu^-},
 \left(\frac{1-\alpha^*}{M_1\max\{1,T\}L_{m-1}}\right)^{1/(1-\alpha^*)},
 \tau^*\right\}.
\]

\begin{theorem}\label{thm:existence}
If $\tau_n\le\tau_{\max}$ for $n=1,\ldots,N$, then the collocation
solution $U\in S_{m-1}^{(-1)}(I_h)$ exists and is unique.
\end{theorem}

\begin{proof}
We argue by induction on $n$. For $n=0$, we assume $U(0)=f(0)$.
Let $n\ge1$, and suppose that $U\in S_{m-1}^{(-1)}$ is the unique solution of
\[
 U(t_{k,i})=f(t_{k,i})-\int_0^{t_{k,i}}
 \frac{U(s)}{(t_{k,i}-s)^{\alpha(U(s))}}\dd s,
 \quad k=0,1,\ldots,n-1,\quad i=1,\ldots,m,
\]
where $t_{k,i}\equiv0$ when $k=0$. We shall prove that there is a unique
function $\widetilde U$ on $(t_{n-1},t_n]$ such that
\vspace{-2mm}
\begin{align*}
 \widetilde U(t_{n,i})
 &=f(t_{n,i})-\int_0^{t_{n-1}}
 \frac{U(s)}{(t_{n,i}-s)^{\alpha(U(s))}}\dd s
 -\int_{t_{n-1}}^{t_{n,i}}
 \frac{\widetilde U(s)}{(t_{n,i}-s)^{\alpha(\widetilde U(s))}}\dd s,
\end{align*}
 for $i=1,2,\ldots,m$.

For $x=(x_1,\ldots,x_m)^\top\in\R^m$, define
$\norm x\coloneqq\max_{1\le i\le m}\abs{x_i}$ and
$B_R(0)\coloneqq\{x\in\R^m:\norm x\le R\}$. For
$x\in B_{M_1}(0)$ and $s\in[t_{n-1},t_n]$, set
$\widetilde U_n(s,x)=L(s)^\top x$, where
\vspace{-2mm}
\[
 L(s)\coloneqq\left(
 \ell_1\left(\frac{s-t_{n-1}}{\tau_n}\right),\ldots,
 \ell_m\left(\frac{s-t_{n-1}}{\tau_n}\right)
 \right)^\top\in\R^m.
\]
By~\eqref{eq:coll1}, then
$\abs{\widetilde U_n(s,x)}\le L_{m-1}\norm x$.

Define $\Delta_n:B_{M_1}(0)\to\R^m$ componentwise by
\vspace{-2mm}
\[
 (\Delta_n(x))_i\coloneqq f(t_{n,i})-
 \int_{0}^{t_{n-1}}
 \frac{U(s)}{(t_{n,i}-s)^{\alpha(U(s))}}\dd s
 -\int_{t_{n-1}}^{t_{n,i}}
 \frac{\widetilde U_n(s,x)}
 {(t_{n,i}-s)^{\alpha(\widetilde U_n(s,x))}}\dd s.
\]
By Lemma~\ref{lem:collocation-stability}, one yields
\vspace{-4mm}
\begin{align*}
 \abs{(\Delta_n(x))_i}
 &\le\norm f+\max\{1,T\}C_0\norm f
 \int_{0}^{t_{n-1}}\frac{\dd s}{(t_{n,i}-s)^{\alpha^*}}\\
 &\qquad+\max\{1,T\}L_{m-1}
 \norm x
 \int_{t_{n-1}}^{t_{n,i}}\frac{\dd s}{(t_{n,i}-s)^{\alpha^*}}\\
 &\le\norm f+\frac{\max\{1,T\}C_0\norm f}{c_1^{\alpha^*}}
 \int_{0}^{t_{n-1}}
 \frac{\dd s}{(t_n-s)^{\alpha^*}}\\
 &\quad+\frac{\max\{1,T\}L_{m-1}\norm x}{1-\alpha^*}
 (t_{n,i}-t_{n-1})^{1-\alpha^*}\\
 &\le\norm f+
 \frac{\max\{1,T\}C_0\norm f\,T^{1-\alpha^*}}
 {c_1^{\alpha^*}(1-\alpha^*)}+\frac{\max\{1,T\}L_{m-1}M_1}{1-\alpha^*}
 \tau_n^{1-\alpha^*}\\
 &=M_0+M_1\frac{\max\{1,T\}L_{m-1}}{1-\alpha^*}
 \tau_n^{1-\alpha^*}\le M_1.
\end{align*}
Thus $\Delta_n$ maps $B_{M_1}(0)$ into itself.

We next prove that $\Delta_n$ is a contraction mapping. For arbitrary
$x,y\in B_{M_1}(0)$,
\begin{align*}
 \abs{(\Delta_n(x)-\Delta_n(y))_i}
 &=\left|\int_{t_{n-1}}^{t_{n,i}}
 \left(
 \frac{\widetilde U_n(s,x)}{(t_{n,i}-s)^{\alpha(\widetilde U_n(s,x))}}
 -\frac{\widetilde U_n(s,y)}{(t_{n,i}-s)^{\alpha(\widetilde U_n(s,y))}}
 \right)\dd s\right|\\
 &\le I_1+I_2,
\end{align*}
where
\begin{align*}
 I_1&\coloneqq\int_{t_{n-1}}^{t_{n,i}}
 \frac{\abs{\widetilde U_n(s,x)-\widetilde U_n(s,y)}}
 {(t_{n,i}-s)^{\alpha(\widetilde U_n(s,x))}}\dd s,\\
 I_2&\coloneqq\int_{t_{n-1}}^{t_{n,i}}
 \abs{\widetilde U_n(s,y)}
 \left|(t_{n,i}-s)^{-\alpha(\widetilde U_n(s,x))}
 -(t_{n,i}-s)^{-\alpha(\widetilde U_n(s,y))}\right|\dd s.
\end{align*}
First,
\vspace{-2mm}
\[
 \abs{\widetilde U_n(s,x)-\widetilde U_n(s,y)}
 =\abs{L(s)^T(x-y)}\le L_{m-1}\norm{x-y}.
\]
Hence $I_1\le L_{m-1}\tau_n^{1-\alpha^*}\norm{x-y}/(1-\alpha^*)$.

For $I_2$, the mean value theorem, $x,y\in B_{M_1}(0)$, and
Assumption~\ref{ass:nonlinear-exponent} imply that there is a number $\xi$ between
$\alpha(\widetilde U_n(s,x))$ and $\alpha(\widetilde U_n(s,y))$ such that
\begin{align*}
 &\left|(t_{n,i}-s)^{-\alpha(\widetilde U_n(s,x))}
 -(t_{n,i}-s)^{-\alpha(\widetilde U_n(s,y))}\right|\\
 &\qquad=\abs{\alpha(\widetilde U_n(s,x))-\alpha(\widetilde U_n(s,y))}
 (t_{n,i}-s)^{-\xi}\abs{\ln(t_{n,i}-s)}\\
 &\qquad\le Q_{L_{m-1}M_1}\abs{\widetilde U_n(s,x)-\widetilde U_n(s,y)}
 (t_{n,i}-s)^{-\alpha^*}\abs{\ln(t_{n,i}-s)}\\
 &\qquad\le Q_{L_{m-1}M_1}L_{m-1}\norm{x-y}(t_{n,i}-s)^{-\nu^+}
 \left[(t_{n,i}-s)^{\nu^-}\abs{\ln(t_{n,i}-s)}\right]\\
 &\qquad\le Q_{L_{m-1}M_1}L_{m-1}
 \max\{T^{\nu^-}\abs{\ln T},(\nu^-e)^{-1}\}
 \norm{x-y}(t_{n,i}-s)^{-\nu^+}.
\end{align*}
Here we used $(t_{n,i}-s)^{\nu^-}\abs{\ln(t_{n,i}-s)}
\le\max\{T^{\nu^-}\abs{\ln T},(\nu^-e)^{-1}\}$.
It follows that
\begin{align*}
 I_2
 &\le M_1Q_{L_{m-1}M_1}L_{m-1}
 \max\{T^{\nu^-}\abs{\ln T},(\nu^-e)^{-1}\}\norm{x-y}
 \int_{t_{n-1}}^{t_{n,i}}(t_{n,i}-s)^{-\nu^+}\dd s\\
 &\le\frac{M_1Q_{L_{m-1}M_1}L_{m-1}
 \max\{T^{\nu^-}\abs{\ln T},(\nu^-e)^{-1}\}}{\nu^-}
 \tau_n^{\nu^-}\norm{x-y}
 =\frac{M_2}{2}\tau_n^{\nu^-}\norm{x-y}.
\end{align*}
Since $\tau_n\le\tau_{\max}\le M_3^{1/\nu^-}$,
\begin{align*}
  \abs{(\Delta_n(x)-\Delta_n(y))_i}
\le\left(\frac{L_{m-1}}{1-\alpha^*}\tau_n^{2\nu^-}
+\frac{M_2}{2}\tau_n^{\nu^-}\right)\norm{x-y}
\le\frac12\norm{x-y},
\end{align*}
which implies that $\norm{\Delta_n(x)-\Delta_n(y)}\le\frac12\norm{x-y}$.
Thus, $\Delta_n:B_{M_1}(0)\to B_{M_1}(0)$ is a contraction mapping.

Therefore, $\Delta_n$ has a unique fixed point
$x^*=(x_1^*,\ldots,x_m^*)^\top\in B_{M_1}(0)$. Defining
$\widetilde U(t_{n,i})=x_i^*$ for $i=1,\ldots,m$ uniquely determines a
function $\widetilde U\in S_{m-1}^{(-1)}$ satisfying the stated conditions.
Moreover, Lemma~\ref{lem:collocation-stability} implies that every
admissible solution satisfies
$\abs{\widetilde U(t_{n,i})}\le M_1$; hence $\widetilde U$ is unique.
This completes the induction.
\end{proof}

\section{Preliminary error estimates}\label{sec:preliminary-error}

To establish the convergence of the proposed scheme, we first present
several preliminary error estimates in this section.

Let $u^I$ denote the Lagrange interpolant of the exact solution $u$ on
$X_h$. Thus, for every $1\le n\le N$ and $s\in(0,1]$,
\[
 u^I(t_{n-1}+s\tau_n)
 =\sum_{j=1}^{m}\ell_j(s)u(t_{n,j}),
 \qquad u^I\in S_{m-1}^{(-1)}.
\]
Define $\rho(t)\coloneqq u(t)-U(t)$ and
$\rho^I(t)\coloneqq u^I(t)-U(t)$ for $t\in(0,T]$, and set
\begin{align*}
 (R_1)_{n,i}&\coloneqq\int_0^{t_{n,i}}
 \frac{u^I(s)-u(s)}{(t_{n,i}-s)^{\alpha(u(s))}}\dd s,\\
 (R_2)_{n,i}&\coloneqq\int_0^{t_{n,i}}U(s)
 \left[\frac1{(t_{n,i}-s)^{\alpha(u^I(s))}}
 -\frac1{(t_{n,i}-s)^{\alpha(u(s))}}\right]\dd s,
 \qquad n=1,2,\ldots,N.
\end{align*}
Finally, let $\rho_n\coloneqq\max_{1\le i\le m}\abs{\rho(t_{n,i})}$.
By \eqref{eq:coll1},
$\norm{\rho^I}_{(t_{n-1},t_n]}\le L_{m-1}\rho_n$.

\begin{lemma}\label{lem:error-propagation}
For all sufficiently large $N$, there is a constant $Q$ such that
\[
 \rho_n\le Q\max_{\substack{1\le d\le n\\1\le i\le m}}
 \left\{\abs{(R_1)_{d,i}}+\abs{(R_2)_{d,i}}\right\}.
\]
\end{lemma}

\begin{proof}
For arbitrary $n\in\{1,2,\ldots,N\}$ and $1\le i\le m$, the exact
equation and the collocation equation give
\[
 \rho(t_{n,i})=-\int_0^{t_{n,i}}
 \left[\frac{u(s)}{(t_{n,i}-s)^{\alpha(u(s))}}
 -\frac{U(s)}{(t_{n,i}-s)^{\alpha(U(s))}}\right]\dd s.
\]
Moreover,
\begin{align*}
\frac{u(s)}{(t_{n,i}-s)^{\alpha(u(s))}}
 &-\frac{U(s)}{(t_{n,i}-s)^{\alpha(U(s))}}
 =\frac{u(s)-u^I(s)}{(t_{n,i}-s)^{\alpha(u(s))}}
 +\frac{\rho^I(s)}{(t_{n,i}-s)^{\alpha(u(s))}}\\
&+U(s)\left[\frac1{(t_{n,i}-s)^{\alpha(u(s))}}
 -\frac1{(t_{n,i}-s)^{\alpha(u^I(s))}}\right] \\
 & +U(s)\left[\frac1{(t_{n,i}-s)^{\alpha(u^I(s))}}
 -\frac1{(t_{n,i}-s)^{\alpha(U(s))}}\right].
\end{align*}
Hence
\begin{align}
 \rho(t_{n,i})
 &=-\int_0^{t_{n,i}}U(s)
 \left[\frac1{(t_{n,i}-s)^{\alpha(u^I(s))}}
 -\frac1{(t_{n,i}-s)^{\alpha(U(s))}}\right]\dd s \notag\\
 &\quad-\int_0^{t_{n,i}}
 \frac{\rho^I(s)}{(t_{n,i}-s)^{\alpha(u(s))}}\dd s
 +(R_1)_{n,i}+(R_2)_{n,i} \notag\\
 &\eqqcolon(\Phi_1)_{n,i}+(\Phi_2)_{n,i}
 +(R_1)_{n,i}+(R_2)_{n,i}.
 \label{eq:err5}
\end{align}

Applying the mean value theorem to the function
$(t_{n,i}-s)^{-\alpha(\cdot)}$ gives
\begin{align*}
 \Big|\frac1{(t_{n,i}-s)^{\alpha(u^I(s))}}
 &-\frac1{(t_{n,i}-s)^{\alpha(U(s))}}\Big|
 =\left|\frac{\ln(t_{n,i}-s)(-\alpha'(\eta(s)))\rho^I(s)}
 {(t_{n,i}-s)^{\alpha(\eta(s))}}\right| \\
 &\le Q\frac{\abs{\ln(t_{n,i}-s)}\abs{\rho^I(s)}}
 {(t_{n,i}-s)^{\alpha^*}} \le Q\frac{\abs{\rho^I(s)}}{(t_{n,i}-s)^{\nu^+}}
 \le Q\frac{\abs{\rho^I(s)}}{(t_n-s)^{\nu^+}},
\end{align*}
where $\eta(s)$ lies between $u^I(s)$ and $U(s)$. Combining this estimate
with the preceding interpolation bound and
Lemma~\ref{lem:collocation-stability} yields
\begin{align*}
 \abs{(\Phi_1)_{n,i}}
 &\le Q\sum_{j=1}^{n-1}\int_{t_{j-1}}^{t_j}
 \rho_j\frac{\dd s}{(t_n-s)^{\nu^+}}
 +Q\int_{t_{n-1}}^{t_{n,i}}
 \rho_n\frac{\dd s}{(t_{n,i}-s)^{\nu^+}}.
\end{align*}
Arguing exactly as in the proof of Lemma~\ref{lem:collocation-stability},
the right-hand side is bounded
by
\[
 Q\tau_n^{\nu^-}\rho_n+Q\sum_{j=1}^{n-1}\rho_j
 \left[(t_n-t_{j-1})^{\nu^-}-(t_n-t_{j})^{\nu^-}\right].
\]
The same argument applies to $\abs{(\Phi_2)_{n,i}}$. Taking $N$ sufficiently
large, and hence $\tau_n$ sufficiently small, combine (\ref{eq:err5}) we obtain
\begin{align*}
 \abs{\rho(t_{n,i})}
 &\le Q\tau_n^{\nu^-}\rho_n+Q\sum_{j=1}^{n-1}\rho_j
 \left[(t_n-t_{j-1})^{\nu^-}-(t_n-t_{j})^{\nu^-}\right]\\
 &\quad+\max_{1\le i\le m}
 \left\{\abs{(R_1)_{n,i}}+\abs{(R_2)_{n,i}}\right\},
 \qquad 1\le i\le m.
\end{align*}
Therefore,
\begin{align*}
 \rho_n
 &\le Q\sum_{j=1}^{n-1}\rho_j
 \left[(t_n-t_{j-1})^{\nu^-}-(t_n-t_{j})^{\nu^-}\right]
 +\max_{1\le i\le m}
 \left\{\abs{(R_1)_{n,i}}+\abs{(R_2)_{n,i}}\right\}\\
 &\le QN^{-\nu^-}\sum_{j=1}^{n-1}\rho_j(n-j)^{-\nu^+}
 +\max_{\substack{1\le d\le n\\1\le i\le m}}
 \left\{\abs{(R_1)_{d,i}}+\abs{(R_2)_{d,i}}\right\}.
\end{align*}
The conclusion follows from Lemma \ref{lem:discrete-gronwall}.
\end{proof}

Based on Lemma \ref{lem:error-propagation}, we give the following corollary.
\begin{corollary}
For all sufficiently large $N$, there exist constants $Q$ and $Q_0$ such
that, for every $n=1,\ldots,N$ and $1\le i\le m$,
\[
 \abs{(R_1)_{n,i}}+\abs{(R_2)_{n,i}}
 \le Q\int_0^{t_{n,i}}
 \frac{\abs{u(s)-u^I(s)}\bigl(\abs{\ln(t_{n,i}-s)}+1\bigr)}
 {(t_{n,i}-s)^{\alpha(u(t_{n,i}))+Q_0\tau^{1-\alpha_0}}}\dd s.
\]
\end{corollary}

\begin{proof}
The argument is identical to that of \cite[Corollary~5.2]{Zheng}. Indeed,
\begin{align*}
 \abs{(R_2)_{n,i}}
 &=\left|\int_0^{t_{n,i}}U(s)
 \left[\frac1{(t_{n,i}-s)^{\alpha(u(s))}}
 -\frac1{(t_{n,i}-s)^{\alpha(u^I(s))}}\right]\dd s\right|\\
 &\le Q\int_0^{t_{n,i}}
 \frac{\abs{u(s)-u^I(s)}\abs{\ln(t_{n,i}-s)}}
 {(t_{n,i}-s)^{\alpha(u(t_{n,i}))+Q_0\tau^{1-\alpha_0}}}\dd s.
\end{align*}
Furthermore,
\[
 \abs{(R_1)_{n,i}}
 \le Q\int_0^{t_{n,i}}
 \frac{\abs{u^I(s)-u(s)}}{(t_{n,i}-s)^{\alpha(u(t_{n,i}))}}\dd s
 \le Q\int_0^{t_{n,i}}
 \frac{\abs{u^I(s)-u(s)}}
 {(t_{n,i}-s)^{\alpha(u(t_{n,i}))+Q_0\tau^{1-\alpha_0}}}\dd s.
\]
This completes the proof.
\end{proof}

\section{Error bounds}\label{sec:error-bounds}

In this section, we shall derive the error bounds of the proposed scheme.

Assume throughout that $\alpha\in C^m(\R)$ and
$f\in C^{m,\alpha_0}(0,T]$. Then, by
 Theorem~\ref{thm:zheng-results}(ii), $u\in C^{m,\alpha_0}(0,T]$.

\begin{lemma}\label{lem:interpolation-remainder}
Let $1\le j\le N$ and $t\in[t_{j-1},t_j]$. Then
\[
 u(t)-u^I(t)=\int_{t_{j-1}}^{t_j}G_j^{(m)}(s,t)u^{(m)}(s)\dd s,
\]
where
\(
 G_j^{(m)}(s,t)\coloneqq\frac{(-1)^m}{(m-1)!}
 \left[(s-t)_+^{m-1}-\sum_{k=1}^{m}
 \ell_k\left(\frac{t-t_{j-1}}{\tau_j}\right)
 (s-t_{j,k})_+^{m-1}\right]
\)
and $x_+\coloneqq\max\{0,x\}$.
\end{lemma}

\begin{proof}
For $j\ge2$, $u\in C^m[t_{j-1},t_j]$, and Taylor's formula with remainder
gives
\begin{equation}
 u(t)=\sum_{i=0}^{m-1}\frac{u^{(i)}(t_j)}{i!}(t-t_j)^i+R_m(t).
 \label{eq:ref6}
\end{equation}
Here
\[
 R_m(t)=\frac1{(m-1)!}\int_{t_j}^t(t-s)^{m-1}u^{(m)}(s)\dd s
 =\frac{(-1)^m}{(m-1)!}\int_{t_{j-1}}^{t_j}
 (s-t)_+^{m-1}u^{(m)}(s)\dd s.
\]
Apply the operator $L_t(v)\coloneqq v(t)-v^I(t)$ to \eqref{eq:ref6}.
Since $L_t(p)=0$ for every polynomial $p$ of degree at most $m-1$,
\begin{align*}
 u(t)-u^I(t)
 &=L_t(u)=L_t(R_m)=\int_{t_{j-1}}^{t_j}L_t\left(
 \frac{(-1)^m}{(m-1)!}(s-t)_+^{m-1}\right)u^{(m)}(s)\dd s\\
 &=\int_{t_{j-1}}^{t_j}G_j^{(m)}(s,t)u^{(m)}(s)\dd s.
\end{align*}

For $j=1$, $u\in C^m(0,t_1]$. For $t\in[0,t_1]$ and $s\in(0,t_1]$,
the inequalities $(s-t)_+\le s$ and
$\abs{u^{(m)}(s)}\le Qs^{1-\alpha_0-m}$ imply
$(s-t)_+^{m-1}\abs{u^{(m)}(s)}
\le Qs^{-\alpha_0}\in L^1(0,t_1)$.
Thus
$R_m(t)\coloneqq\frac{(-1)^m}{(m-1)!}\int_0^{t_1}
(s-t)_+^{m-1}u^{(m)}(s)\dd s$
converges absolutely for $t\in[0,t_1]$. Formula \eqref{eq:ref6} remains
valid for $t\in(0,t_1]$ because $u\in C^m(0,t_1]$. The dominated
convergence theorem further gives
$\lim_{t\to0^+}R_m(t)=R_m(0)$, so \eqref{eq:ref6} holds for every
$t\in[0,t_1]$. The same argument as for $j\ge2$ then yields
$u(t)-u^I(t)=\int_0^{t_1}G_1^{(m)}(s,t)u^{(m)}(s)\dd s$.
\end{proof}

\begin{remark}
The form of $G_j^{(m)}$ shows that
\[
 \abs{G_j^{(m)}(s,t)}
 \le\frac1{(m-1)!}\left(1+\sum_{k=1}^{m}
 \left|\ell_k\left(\frac{t-t_{j-1}}{\tau_j}\right)\right|\right)
 (s-t_{j-1})^{m-1}
 \le Q(s-t_{j-1})^{m-1}.
\]
\end{remark}

Next, we provide several auxiliary estimates that will be used in the
subsequent convergence analysis.

\begin{lemma}\label{lem:remainder-bound}
For $j=1,\ldots,N$, suppose that
$\abs{u(s)-u^I(s)}\le\eta(s)$ for $s\in(t_{j-1},t_j)$, where
$\eta=\eta_j$ is constant on $(t_{j-1},t_j)$ and
$\eta_j\ge\eta_{j+1}$. Then there is a constant $Q$ such that
\[
 \abs{(R_1)_{n,i}}+\abs{(R_2)_{n,i}}
 \le Qt_{n,i}^{-\alpha_0}(1+\abs{\ln t_{n,i}})
 \int_0^{t_{n,i}}\eta(s)\dd s,
\]
for $n=1,\ldots,N$ and $i=1,\ldots,m$.
\end{lemma}

\begin{proof}
The proof follows from \cite[Lemma~6.1]{Zheng} and is thus omitted. 
\end{proof}

Based on the preceding
estimates, we obtain the following convergence results.
\begin{theorem}\label{thm6.3}
Suppose that Assumption~\ref{ass:nonlinear-exponent} holds and that
$\tau_n\le\tau_{\max}$ for $n=1,\ldots,N$.  Then there exists a constant
$Q$, independent of $n$ and $N$, such that the following estimates hold for
$1\le n\le N$ and $m\ge3$:

\smallskip
\noindent\textbf{Case 1.} If $f\in C^{m,\alpha_0}(0,T]$ and $\alpha_0>0$,
then
\[
 \rho_n\le Q\max_{\substack{1\le d\le n\\1\le i\le m}}
 (1+\abs{\ln t_{d,i}})
 \begin{cases}
 N^{-2r(1-\alpha_0)}d^{r(1-\alpha_0)-1},
     &r(1-\alpha_0)<m-1,\\
 N^{-m}t_{d,i}^{(m-2)/r}(1+\ln {d}),
     &r(1-\alpha_0)=m-1,\\
 N^{-m}t_{d,i}^{[2r(1-\alpha_0)-m]/r},
     &m-1<r(1-\alpha_0)\le m.
 \end{cases}
\]

\noindent\textbf{Case 2.} If $f\in C^{m,\alpha_0}(0,T]$ and $\alpha_0=0$,
then
\[
 \rho_n\le Q\max_{\substack{1\le d\le n\\1\le i\le m}}
 (1+\abs{\ln t_{d,i}})
 \begin{cases}
 N^{-2r}d^{r-1},&1\le r<m-1,\\
 N^{-m}t_{d,i}^{(m-2)/r}(1+\ln {d}),&r=m-1,\\
 N^{-m}t_{d,i}^{(2r-m)/r},&m-1<r\le m.
 \end{cases}
\]

\noindent\textbf{Case 3.} If $f\in C^m[0,T]$ and
$m$th-order smoothing condition~\eqref{eq:reg4} is satisfied,
then
\[
 \rho_n\le Q\max_{\substack{1\le d\le n\\1\le i\le m}}
 N^{-m}(1+\abs{\ln t_{d,i}})t_{d,i}^{1-\alpha_0},
 \qquad r=1.
\]
\end{theorem}

\begin{proof}
The proof can be performed similarly as
\cite[Theorem~6.3]{Zheng}, with
Lemmas~\ref{lem:interpolation-remainder}--\ref{lem:remainder-bound} and Theorem \ref{thm:volterra-regularity}
providing the required estimates for general $m$; and is thus omitted. 
\end{proof}

\vskip 1mm
We next simplify the collocation-point estimates in
Theorem~\ref{thm6.3}.
\begin{corollary}\label{cor6.4}
Suppose that the assumptions of Theorem~\ref{thm6.3} hold. Then, for $m\ge3$ and $1\le n\le N$, the following estimates hold:

\smallskip
\noindent\textbf{Case 1.} If $f\in C^{m,\alpha_0}(0,T]$ and
$\alpha_0>0$, then
\[
 \rho_n\le Q
 \begin{cases}
 N^{-2r(1-\alpha_0)}\ln N,
     &0<r(1-\alpha_0)\le1,\\
 N^{-r(1-\alpha_0)-1},
     &1<r(1-\alpha_0)<m-1,\\
 N^{-m}\ln N,
     &r(1-\alpha_0)=m-1,\\
 N^{-m},
     &m-1<r(1-\alpha_0)\le m.
 \end{cases}
\]

\noindent\textbf{Case 2.} If $f\in C^{m,\alpha_0}(0,T]$ and
$\alpha_0=0$, then
\[
 \rho_n\le Q
 \begin{cases}
 N^{-2}\ln N,&r=1,\\
 N^{-r-1},&1<r<m-1,\\
 N^{-m}\ln N,&r=m-1,\\
 N^{-m},&m-1<r\le m.
 \end{cases}
\]

\noindent\textbf{Case 3.} If $f\in C^m[0,T]$ and the
$m$th-order smoothing condition \eqref{eq:reg4} is satisfied, then
\[
 \rho_n\le QN^{-m},\qquad r=1.
\]
\end{corollary}

\begin{proof}
We first consider Case~1. Since $t_{d,i}\sim t_d=d^r\tau^r\sim d^rN^{-r}$, for
$r(1-\alpha_0)<m-1$ the first estimate in
Theorem~\ref{thm6.3} gives
\[
 \begin{aligned}
 &(1+\abs{\ln t_{d,i}})N^{-2r(1-\alpha_0)}
 d^{r(1-\alpha_0)-1}\le QN^{-2r(1-\alpha_0)}d^{r(1-\alpha_0)-1}
 \left(1+\ln\frac Nd\right).
 \end{aligned}
\]
If $0<r(1-\alpha_0)\le1$, then
$d^{r(1-\alpha_0)-1}\le1$, and hence
$\rho_n\le QN^{-2r(1-\alpha_0)}\ln N$.

If $1<r(1-\alpha_0)<m-1$, put $q=d/N$. Then
\[
 \begin{aligned}
 &d^{r(1-\alpha_0)-1}\left(1+\ln\frac Nd\right)
 =N^{r(1-\alpha_0)-1}
 q^{r(1-\alpha_0)-1}(1+\abs{\ln q})
 \le QN^{r(1-\alpha_0)-1},
 \end{aligned}
\]
because $q^{r(1-\alpha_0)-1}(1+\abs{\ln q})\le Q$ for $0<q\le1$.
It follows that $\rho_n\le QN^{-r(1-\alpha_0)-1}$.

If $r(1-\alpha_0)=m-1$, then $(m-2)/r>0$ and
$t_{d,i}^{(m-2)/r}(1+\abs{\ln t_{d,i}})\le Q$.
Thus the second estimate in Theorem~\ref{thm6.3}, together with
$1+\ln d\le Q\ln N$, gives $\rho_n\le QN^{-m}\ln N$.

If $m-1<r(1-\alpha_0)\le m$, then
$[2r(1-\alpha_0)-m]/r>0$ and
$t_{d,i}^{[2r(1-\alpha_0)-m]/r}(1+\abs{\ln t_{d,i}})\le Q$.
Consequently, $\rho_n\le QN^{-m}$.

Case~2 follows by setting $\alpha_0=0$. In particular, the four ranges
become $r=1$, $1<r<m-1$, $r=m-1$, and $m-1<r\le m$, respectively.
For Case~3, $t^{1-\alpha_0}(1+\abs{\ln t})$ is bounded on $(0,T]$;
hence the last estimate in Theorem~\ref{thm6.3} gives
$\rho_n\le QN^{-m}$ for $r=1$.
\end{proof}

\vskip 1mm
We next give the following global error estimates on $[0,T]$.
\begin{corollary}\label{cor6.5}
Suppose that $m\ge3$ and the same assumptions of Theorem~\ref{thm6.3} hold. There
exists a constant $Q$, independent of $n$ and $N$, such that:

\smallskip
\noindent\textbf{Case 1.} If $f\in C^{m,\alpha_0}(0,T]$ and $\alpha_0>0$,
then $\sup_{t\in[0,T]}\abs{U(t)-u(t)}\le QN^{-r(1-\alpha_0)}$ for
$1\le r\le m/(1-\alpha_0)$.

\noindent\textbf{Case 2.} If $f\in C^{m,\alpha_0}(0,T]$ and $\alpha_0=0$,
then $\sup_{t\in[0,T]}\abs{U(t)-u(t)}\le QN^{-r}$ for
$1\le r\le m$.

\noindent\textbf{Case 3.} If $f\in C^m[0,T]$ satisfies the
$m$th-order smoothing condition \eqref{eq:reg4}, then
$\sup_{t\in[0,T]}\abs{U(t)-u(t)}\le QN^{-m}$ for $r=1$.
\end{corollary}

\begin{proof} We derive the corollary by the following three cases.

\textbf{Case 1:}
For $t\in[t_{n-1},t_n]$,
\begin{align}
 \abs{u(t)-U(t)}
 &\le\abs{u(t)-u^I(t)}+\abs{u^I(t)-U(t)}\notag\\
 &\le Q\left\{
 \left[n^{r(1-\alpha_0)/m}-(n-1)^{r(1-\alpha_0)/m}\right]^m
 \tau^{r(1-\alpha_0)}+\rho_n\right\}.
 \label{eq:ref11}
\end{align}
Moreover,
\(
 \left[n^{r(1-\alpha_0)/m}-(n-1)^{r(1-\alpha_0)/m}\right]^m
 \tau^{r(1-\alpha_0)}
 \le\tau^{r(1-\alpha_0)}=QN^{-r(1-\alpha_0)}.
\)
By Theorem~\ref{thm6.3}, we have
\[
 \begin{aligned}
 \rho_n
 &\le Q\max_{\substack{1\le d\le n\\1\le i\le m}}
 (1+\abs{\ln t_{d,i}})\begin{cases}
 N^{-2r(1-\alpha_0)}d^{r(1-\alpha_0)-1},
     &r(1-\alpha_0)<m-1,\\
 N^{-m}t_{d,i}^{(m-2)/r}(1+\ln d),
     &r(1-\alpha_0)=m-1,\\
 N^{-m}t_{d,i}^{[2r(1-\alpha_0)-m]/r},
     &m-1<r(1-\alpha_0)\le m.
 \end{cases}
 \end{aligned}
\]

Since $t_{d,i}\sim(d/N)^r$, if
$r(1-\alpha_0)<m-1$, then
\begin{align*}
 &(1+\abs{\ln t_{d,i}})N^{-2r(1-\alpha_0)}d^{r(1-\alpha_0)-1}\\
 &\qquad\le QN^{-2r(1-\alpha_0)}N^{r(1-\alpha_0)-1}
 t_{d,i}^{1-\alpha_0-1/r}(1+\abs{\ln t_{d,i}})\\
 &\qquad\le QN^{-r(1-\alpha_0)-1}t_{d,i}^{-1/r}
 \le QN^{-r(1-\alpha_0)-1}d^{-1}N
 \le QN^{-r(1-\alpha_0)}.
\end{align*}
If $r(1-\alpha_0)=m-1$, then
\[
 (1+\abs{\ln t_{d,i}})N^{-m}t_{d,i}^{(m-2)/r}(1+\ln d)
 \le QN^{-r(1-\alpha_0)-1}(1+\ln d)
 \le QN^{-r(1-\alpha_0)},
\]
because $(m-2)/r\ge1/r>0$. Finally, if
$m-1<r(1-\alpha_0)\le m$, then
\[
 (1+\abs{\ln t_{d,i}})N^{-m}
 t_{d,i}^{[2r(1-\alpha_0)-m]/r}
 \le QN^{-m}\le QN^{-r(1-\alpha_0)},
\]
because $2r(1-\alpha_0)-m\ge m-2$.

Consequently, for $1\le r\le m/(1-\alpha_0)$,
$\rho_n \le QN^{-r(1-\alpha_0)}$. Equation
\eqref{eq:ref11} therefore yields
$\sup_{t\in[0,T]}\abs{U(t)-u(t)}\le QN^{-r(1-\alpha_0)}$.

\textbf{Case 2:}
The proof is identical to that of Case~1 after setting $\alpha_0=0$.

\textbf{Case 3:}
Here $r=1$ and $u\in C^m[0,T]$. By
Lemma~\ref{lem:interpolation-remainder},
$\abs{u(t)-u^I(t)}\le QN^{-m}$. Moreover, we have
\[
 \abs{u^I(t)-U(t)}
 \le Q\rho_n
 \le Q\max_{\substack{1\le d\le n\\1\le i\le m}}
 N^{-m}(1+\abs{\ln t_{d,i}})t_{d,i}^{1-\alpha_0}
 \le QN^{-m}.
\]
Thus
$\sup_{t\in[0,T]}\abs{U(t)-u(t)}\le QN^{-m}$. We then finish the proof.
\end{proof}

\section{Numerical experiments}

\begingroup
\newcommand{\sci}[2]{#1\!\times\!10^{\number#2}}
\newcommand{\ee}[2]{#1\!\times\!10^{\number#2}}
\setcounter{table}{0}
\renewcommand{\thetable}{\thesection.\arabic{table}}

In this section, we set out several numerical experiments that validate
the error bounds in Theorem~\ref{thm6.3} and
Corollaries~\ref{cor6.4}--\ref{cor6.5} for the cases $m=3,4$.

\medskip
\noindent\textbf{Example 1.}
For $m=3$, we test the Gauss collocation parameters
$c_1=\frac{5-\sqrt{15}}{10}$, $c_2=\frac12$, and
$c_3=\frac{5+\sqrt{15}}{10}$; the Radau IIA collocation parameters
$c_1=\frac{4-\sqrt6}{10}$, $c_2=\frac{4+\sqrt6}{10}$, and $c_3=1$;
and two sets of arbitrary collocation parameters: $c_1=\frac13$,
$c_2=\frac12$, and $c_3=1$; $c_1=\frac13$, $c_2=\frac12$, and $c_3=\frac23$.

To solve the nonlinear scheme, the initial value $U(t_0)=f(0)$ is known.
For each $n\geq1$, we employ the fixed-point iteration induced by the
contraction mapping $\Delta_n$ in Theorem~\ref{thm:existence}. Let
\[
 \mathbf U_n^{(k)}
 :=\bigl(U^{(k)}(t_{n,1}),\!U^{(k)}(t_{n,2}),\!U^{(k)}(t_{n,3})\bigr)^\top, \ U^{(k)}(s):=\sum_{d=1}^{3}\ell_d\!\left(\frac{s-t_{n-1}}{\tau_n}\right)U^{(k)}(t_{n,d}).
\]
For $i=1,2,3$, define
\begingroup
\setlength{\abovedisplayskip}{2pt}
\setlength{\belowdisplayskip}{2pt}
\setlength{\abovedisplayshortskip}{2pt}
\setlength{\belowdisplayshortskip}{2pt}
\setlength{\jot}{1pt}
\begin{align*}
 U^{(k+1)}(t_{n,i})
 &:=\bigl(\Delta_n(\mathbf U_n^{(k)})\bigr)_i=f(t_{n,i})-
 \sum_{\ell=1}^{n-1}\int_{t_{\ell-1}}^{t_\ell}
 \frac{U(s)}{(t_{n,i}-s)^{\alpha(U(s))}}\,\mathrm ds\\
 &\quad-
 \int_{t_{n-1}}^{t_{n,i}}
 \frac{U^{(k)}(s)}
 {(t_{n,i}-s)^{\alpha(U^{(k)}(s))}}\,\mathrm ds.
\end{align*}
\endgroup
When $n=1$, the sum over $\ell$ is understood to be zero. The stopping
criterion is
\[
 \left\|\mathbf U_n^{(k+1)}-\mathbf U_n^{(k)}\right\|_\infty
 =\max_{1\leq i\leq3}\left|U^{(k+1)}(t_{n,i})-U^{(k)}(t_{n,i})\right|
 \leq2\times10^{-14}.
\]
The integrals are evaluated by adaptive Gauss--Kronrod quadrature.

In each example, the exact solution is unknown; thus, we use the double-mesh
principle \cite{Farr} to estimate the errors in the computed solutions. Let $U^N$ denote
the computed solution on the mesh $\{t_0,t_1,\ldots,t_N\}$, and let
$\widehat U^{2N}$ denote the computed solution on the mesh
$\{\widehat t_0,\widehat t_1,\ldots,\widehat t_{2N}\}$. According to the mesh
definition in Section~\ref{sec:collocation-existence}, $\widehat t_{2j}=t_j$ for
$j=0,1,\ldots,N$. We measure
$E(N):=\max_{0\leq j\leq N}\lvert U^N(t_j)-\widehat
U^{2N}(\widehat t_{2j})\rvert$ and compute the associated numerical rate by
$\text{Numerical Rate}:=\log_2\!\bigl(E(N)/E(2N)\bigr)$.

\setlength{\parindent}{0pt}
\medskip

\noindent\textbf{Case 1 test for uniform meshes.}
Let $T=1$ and $f(t)\equiv1$. For $r=1$, we test two choices of
$\alpha(u(t))$ that satisfy $\alpha_0>0$. For
$\alpha(u)=1/(u^2+1.5)$ one has $\alpha_0=2/5$, whereas for the second
choice one has $\alpha_0=1/14$. Corollary~\ref{cor6.5} predicts the basic
order $1-\alpha_0$, while Corollary~\ref{cor6.4} predicts the enhanced order
$2(1-\alpha_0)$ for the corresponding
superconvergent collocation parameters. The results in Tables~\ref{tab:m3-uniform-alpha-1} and~\ref{tab:m3-uniform-alpha-2}
show that the Gauss and $(1/3,1/2,2/3)$ orders are close to the basic order,
while the orders of Radau IIA and $(1/3,1/2,1)$ with the collocation parameter $c_m=1$ (then the mesh points $t_n$ are also collocation points) exhibit the predicted higher
convergence behaviour.

\begin{table}[H]
\centering \footnotesize
\caption{Errors and convergence rates for $m=3$, $\alpha(u)=1/(u^2+1.5)$, and $r=1$.}
\label{tab:m3-uniform-alpha-1}
\setlength{\tabcolsep}{6pt}
\begin{tabular}{@{}rcccc@{}}
\toprule
$N$ & Gauss & Radau IIA & $(1/3,1/2,1)$ & $(1/3,1/2,2/3)$ \\
\midrule
1024 & $\sci{3.0631}{-04}$ & $\sci{1.6884}{-06}$ & $\sci{6.1355}{-06}$ & $\sci{5.2246}{-04}$ \\
2048 & $\sci{2.0359}{-04}$ & $\sci{8.2395}{-07}$ & $\sci{3.0343}{-06}$ & $\sci{3.4686}{-04}$ \\
4096 & $\sci{1.3506}{-04}$ & $\sci{3.9791}{-07}$ & $\sci{1.4800}{-06}$ & $\sci{2.2995}{-04}$ \\
8192 & $\sci{8.9476}{-05}$ & $\sci{1.9054}{-07}$ & $\sci{7.1424}{-07}$ & $\sci{1.5227}{-04}$ \\
\midrule
Numer. Rate & 0.5941 & 1.0624 & 1.0511 & 0.5947 \\
\bottomrule
\end{tabular}
\end{table}

\begin{table}[H]
\centering \footnotesize
\caption{Errors and convergence rates for $m=3$, $\alpha(u)=[u-\tfrac12u(0)]^2/\{2[u^2+\tfrac14u^2(0)]+1\}$, and $r=1$.}
\label{tab:m3-uniform-alpha-2}
\setlength{\tabcolsep}{6pt}
\begin{tabular}{@{}rcccc@{}}
\toprule
$N$ & Gauss & Radau IIA & $(1/3,1/2,1)$ & $(1/3,1/2,2/3)$ \\
\midrule
256 & $\sci{1.2655}{-05}$ & $\sci{1.8754}{-08}$ & $\sci{1.1197}{-07}$ & $\sci{2.3131}{-05}$ \\
512 & $\sci{6.5722}{-06}$ & $\sci{5.5354}{-09}$ & $\sci{3.2837}{-08}$ & $\sci{1.1994}{-05}$ \\
1024 & $\sci{3.4314}{-06}$ & $\sci{1.6242}{-09}$ & $\sci{9.6034}{-09}$ & $\sci{6.2566}{-06}$ \\
2048 & $\sci{1.7967}{-06}$ & $\sci{4.7443}{-10}$ & $\sci{2.8005}{-09}$ & $\sci{3.2744}{-06}$ \\
\midrule
Numer. Rate & 0.9334 & 1.7755 & 1.7778 & 0.9341 \\
\bottomrule
\end{tabular}
\end{table}

\medskip
\noindent\textbf{Case 2 test for nonuniform meshes.}
Let $T=1$ and $f(t)=t\log t-t$, so $u(0)=f(0)=0$. We test
 $\alpha(u(t))=\frac{u^2(t)}{u^2(t)+3} \ \text{and}\ \alpha(u(t))=\frac{\sin^2(u(t))}{2u^2(t)+3}.$
For both choices, $\alpha_0=\alpha(u(0))=0$. Hence the critical grading
exponent for $m=3$ is $r=(m-1)/(1-\alpha_0)=2$. Corollary~\ref{cor6.4} gives the
critical-mesh estimate rate near or greater than 3 when
$r=2/(1-\alpha_0)$, it is supported from Tables~\ref{tab:m3-graded-alpha-1} and~\ref{tab:m3-graded-alpha-2} that the Gauss
and $(1/3,1/2,2/3)$ results display the expected $r$-order behaviour,
while the Radau IIA and $(1/3,1/2,1)$ choices produce higher terminal
orders.

\begin{table}[H]
\centering \footnotesize
\caption{Errors and convergence rates for $m=3$, $\alpha_0=0$, $\alpha(u)=u^2/(u^2+3)$, and $r=2$.}
\label{tab:m3-graded-alpha-1}
\setlength{\tabcolsep}{6pt}
\begin{tabular}{@{}rcccc@{}}
\toprule
$N$ & Gauss & Radau IIA & $(1/3,1/2,1)$ & $(1/3,1/2,2/3)$ \\
\midrule
256 & $\sci{3.9341}{-07}$ & $\sci{4.7187}{-12}$ & $\sci{3.0421}{-09}$ & $\sci{8.2322}{-07}$ \\
512 & $\sci{9.8348}{-08}$ & $\sci{3.1752}{-13}$ & $\sci{3.8236}{-10}$ & $\sci{2.0580}{-07}$ \\
1024 & $\sci{2.4587}{-08}$ & $\sci{2.1649}{-14}$ & $\sci{4.7941}{-11}$ & $\sci{5.1449}{-08}$ \\
\midrule
Numer. Rate & 2.0000 & 3.8745 & 2.9956 & 2.0000 \\
\bottomrule
\end{tabular}
\end{table}

\begin{table}[H]
\centering \footnotesize
\caption{Errors and convergence rates for $m=3$, $\alpha_0=0$, $\alpha(u)=\sin^2u/(2u^2+3)$, and $r=2$.}
\label{tab:m3-graded-alpha-2}
\setlength{\tabcolsep}{6pt}
\begin{tabular}{@{}rcccc@{}}
\toprule
$N$ & Gauss & Radau IIA & $(1/3,1/2,1)$ & $(1/3,1/2,2/3)$ \\
\midrule
256 & $\sci{3.9341}{-07}$ & $\sci{3.0675}{-12}$ & $\sci{2.8943}{-09}$ & $\sci{8.2322}{-07}$ \\
512 & $\sci{9.8348}{-08}$ & $\sci{2.0206}{-13}$ & $\sci{3.6340}{-10}$ & $\sci{2.0580}{-07}$ \\
1024 & $\sci{2.4587}{-08}$ & $\sci{1.3767}{-14}$ & $\sci{4.5536}{-11}$ & $\sci{5.1449}{-08}$ \\
\midrule
Numer. Rate & 2.0000 & 3.8755 & 2.9965 & 2.0000 \\
\bottomrule
\end{tabular}
\end{table}

\medskip
\noindent\textbf{Case 3 test for uniform meshes.}
Let $T=1$, $f(t)\equiv1$, and $r=1$. We test
$\alpha(u)=\frac12\bigl(1-\exp(-(u-u(0))^4)\bigr)$ and
$\alpha(u)=\frac12 (u-u(0))^4/[1+(u-u(0))^4]$, which satisfies $u(0)=1$, then both choices
satisfy $\alpha(u(0))=\alpha'(u(0))=\alpha''(u(0))=0$. The results in
Tables~\ref{tab:m3-case3-exp-quartic} and
\ref{tab:m3-case3-rational-quartic} support the third-order estimate in
Case~3 of Corollary~\ref{cor6.4}.

\begin{table}[H]
\centering \footnotesize
\caption{Errors and convergence rates for $m=3$, $\alpha(u)=\frac12\bigl(1-\exp(-(u-u(0))^4)\bigr)$, and $r=1$.}
\label{tab:m3-case3-exp-quartic}
\setlength{\tabcolsep}{6pt}
\begin{tabular}{@{}rcccc@{}}
\toprule
$N$ & Gauss & Radau IIA & $(1/3,1/2,1)$ & $(1/3,1/2,2/3)$ \\
\midrule
16  & $\sci{1.8676}{-6}$ & $\sci{6.1691}{-10}$ & $\sci{3.5489}{-7}$  & $\sci{4.1467}{-6}$ \\
32  & $\sci{2.3282}{-7}$ & $\sci{3.5621}{-11}$ & $\sci{4.4682}{-8}$  & $\sci{5.1716}{-7}$ \\
64  & $\sci{2.9048}{-8}$ & $\sci{2.1380}{-12}$ & $\sci{5.6038}{-9}$  & $\sci{6.4537}{-8}$ \\
128 & $\sci{3.6273}{-9}$ & $\sci{1.3234}{-13}$ & $\sci{7.0158}{-10}$ & $\sci{8.0597}{-9}$ \\
\midrule
Numer. Rate & 3.0015 & 4.0139 & 2.9977 & 3.0013 \\
\bottomrule
\end{tabular}
\end{table}

\begin{table}[H]
\centering \footnotesize
\caption{Errors and convergence rates for $m=3$, $\alpha(u)=\frac12 (u-u(0))^4/[1+(u-u(0))^4]$, and $r=1$.}
\label{tab:m3-case3-rational-quartic}
\setlength{\tabcolsep}{6pt}
\begin{tabular}{@{}rcccc@{}}
\toprule
$N$ & Gauss & Radau IIA & $(1/3,1/2,1)$ & $(1/3,1/2,2/3)$ \\
\midrule
16  & $\sci{1.8669}{-6}$ & $\sci{5.7098}{-10}$ & $\sci{3.4875}{-7}$  & $\sci{4.1453}{-6}$ \\
32  & $\sci{2.3277}{-7}$ & $\sci{3.2730}{-11}$ & $\sci{4.3935}{-8}$  & $\sci{5.1706}{-7}$ \\
64  & $\sci{2.9041}{-8}$ & $\sci{1.9502}{-12}$ & $\sci{5.5126}{-9}$  & $\sci{6.4521}{-8}$ \\
128 & $\sci{3.6265}{-9}$ & $\sci{1.1946}{-13}$ & $\sci{6.9019}{-10}$ & $\sci{8.0581}{-9}$ \\
\midrule
Numer. Rate & 3.0014 & 4.0290 & 2.9977 & 3.0013 \\
\bottomrule
\end{tabular}
\end{table}

\setlength{\tabcolsep}{6pt}
\renewcommand{\arraystretch}{1.00}
\medskip
\noindent\textbf{Example 2.}
For $m=4$, the algorithm is the same as for $m=3$, and we use four similar
collocation parameter sets: the four-point Gauss and Radau IIA sets, together
with the two fixed sets $(1/4,1/2,3/4,1)$ and $(1/5,2/5,3/5,4/5)$.

\medskip
\noindent\textbf{Case 1 test for nonuniform meshes.}
Let $T=0.5$ and $f(t)\equiv1$. We test
$\alpha(u)=2u^2/(3+2u^2)$, for which $\alpha_0=2/5$, and
$\alpha(u)=\sin^2u/(1+\sin^2u)$, for which
$\alpha_0=\sin^2(1)/(1+\sin^2(1))$.
Tables~\ref{tab:m4-f1-alpha-1} and~\ref{tab:m4-f1-alpha-2} give rate $3.0000$ for Gauss and $(1/5,2/5,3/5,4/5)$,
supporting the global estimate in Corollary~\ref{cor6.5}, while the higher rates of Radau IIA and
$(1/4,1/2,3/4,1)$ display additional collocation-dependent superconvergence as predicted in Corollary~\ref{cor6.4}.

\begin{table}[H]
\centering \footnotesize
\caption{Errors and convergence rates for $m=4$, $\alpha(u)=2u^2/(3+2u^2)$, $\alpha_0=2/5$, and $r=5$.}
\label{tab:m4-f1-alpha-1}
\begin{tabular}{@{}rcccc@{}}
\toprule
$N$ & Gauss & Radau IIA & $(1/4,1/2,3/4,1)$ & $(1/5,2/5,3/5,4/5)$ \\
\midrule
128 & $\ee{2.1536}{-8}$ & $\ee{6.0333}{-12}$ & $\ee{3.1993}{-10}$ & $\ee{5.2160}{-8}$ \\
256 & $\ee{2.6917}{-9}$ & $\ee{2.6046}{-13}$ & $\ee{2.1343}{-11}$ & $\ee{6.5196}{-9}$ \\
512 & $\ee{3.3645}{-10}$ & $\ee{1.1324}{-14}$ & $\ee{1.3957}{-12}$ & $\ee{8.1494}{-10}$ \\
\midrule
Numer. Rate & 3.0000 & 4.5236 & 3.9348 & 3.0000 \\
\bottomrule
\end{tabular}
\end{table}

\vspace{4mm}

\begin{table}[H]
\centering \footnotesize
\caption{Errors and convergence rates for $m=4$,
$\alpha(u)=\sin^2u/(1+\sin^2u)$,
$\alpha_0=\sin^2(1)/(1+\sin^2(1))$, and
$r=3(1+\sin^2(1))\approx5.1242203$.}
\label{tab:m4-f1-alpha-2}
\begin{tabular}{@{}rcccc@{}}
\toprule
$N$ & Gauss & Radau IIA & $(1/4,1/2,3/4,1)$ & $(1/5,2/5,3/5,4/5)$ \\
\midrule
128 & $\ee{2.4682}{-8}$ & $\ee{7.2792}{-12}$ & $\ee{3.5327}{-10}$ & $\ee{5.9459}{-8}$ \\
256 & $\ee{3.0850}{-9}$ & $\ee{3.1442}{-13}$ & $\ee{2.3693}{-11}$ & $\ee{7.4321}{-9}$ \\
512 & $\ee{3.8562}{-10}$ & $\ee{1.3545}{-14}$ & $\ee{1.5552}{-12}$ & $\ee{9.2901}{-10}$ \\
\midrule
Numer. Rate & 3.0000 & 4.5369 & 3.9293 & 3.0000 \\
\bottomrule
\end{tabular}
\end{table}

\medskip
\noindent\textbf{Case 2 test for nonuniform meshes.}
Let $T=1$ and $f(t)=t\log t-t$. We test
$\alpha(u)=u^2/(u^2+3)$ and $\alpha(u)=\sin^2u/(2u^2+3)$. For both choices,
$\alpha_0=\alpha(u(0))=0$. Tables~\ref{tab:m4-alpha-1} and~\ref{tab:m4-alpha-2} give rate $3.0000$ for Gauss
and $(1/5,2/5,3/5,4/5)$, independently supporting the theorem. The rates near $4.8$ for Radau IIA and near $4$ for
$(1/4,1/2,3/4,1)$ reveal the additional right-endpoint superconvergence.

\begin{table}[H]
\centering \footnotesize
\caption{Errors and convergence rates for $m=4$, $\alpha(u)=u^2/(u^2+3)$,  $\alpha_0=0$,
and $r=3$.}
\label{tab:m4-alpha-1}
\begin{tabular}{@{}rcccc@{}}
\toprule
$N$ & Gauss & Radau IIA & $(1/4,1/2,3/4,1)$ & $(1/5,2/5,3/5,4/5)$ \\
\midrule
64 & $\ee{9.4295}{-8}$ & $\ee{1.1508}{-11}$ & $\ee{3.5300}{-9}$ & $\ee{2.1956}{-7}$ \\
128 & $\ee{1.1787}{-8}$ & $\ee{4.0112}{-13}$ & $\ee{2.2887}{-10}$ & $\ee{2.7445}{-8}$ \\
256 & $\ee{1.4733}{-9}$ & $\ee{1.3878}{-14}$ & $\ee{1.4607}{-11}$ & $\ee{3.4306}{-9}$ \\
\midrule
Numer. Rate & 3.0000 & 4.8532 & 3.9697 & 3.0000 \\
\bottomrule
\end{tabular}
\end{table}

\vspace{4mm}

\begin{table}[H]
\centering \footnotesize
\caption{Errors and convergence rates for $m=4$, $\alpha(u)=\sin^2u/(2u^2+3)$, $\alpha_0=0$,
and $r=3$.}
\label{tab:m4-alpha-2}
\begin{tabular}{@{}rcccc@{}}
\toprule
$N$ & Gauss & Radau IIA & $(1/4,1/2,3/4,1)$ & $(1/5,2/5,3/5,4/5)$ \\
\midrule
64 & $\ee{9.4295}{-8}$ & $\ee{5.9444}{-12}$ & $\ee{3.4118}{-9}$ & $\ee{2.1956}{-7}$ \\
128 & $\ee{1.1787}{-8}$ & $\ee{2.0461}{-13}$ & $\ee{2.1898}{-10}$ & $\ee{2.7445}{-8}$ \\
256 & $\ee{1.4733}{-9}$ & $\ee{7.2164}{-15}$ & $\ee{1.3891}{-11}$ & $\ee{3.4306}{-9}$ \\
\midrule
Numer. Rate & 3.0000 & 4.8255 & 3.9786 & 3.0000 \\
\bottomrule
\end{tabular}
\end{table}

\medskip
\noindent\textbf{Case 3 test for uniform meshes.}
Let $T=1$, $f(t)\equiv1$, and $r=1$. We test
$\alpha(u)=\frac23\bigl(1-\exp(-(u-u(0))^4)\bigr)$ and
$\alpha(u)=\frac23 (u-u(0))^4/[1+(u-u(0))^4]$, which satisfies $u(0)=1$, and both choices
satisfy $\alpha(u(0))=\alpha'(u(0))=\alpha''(u(0))=\alpha'''(u(0))=0$. The results in
Tables~\ref{tab:m4-case3-exp-quartic} and
\ref{tab:m4-case3-rational-quartic} support the fourth-order estimate in
Case~3 of Corollary~\ref{cor6.4}.

\begin{table}[H]
\centering \footnotesize
\caption{Errors and convergence rates for $m=4$, $\alpha(u)=\frac23\bigl(1-\exp(-(u-u(0))^4)\bigr)$, and $r=1$.}
\label{tab:m4-case3-exp-quartic}
\begin{tabular}{@{}rcccc@{}}
\toprule
$N$ & Gauss & Radau IIA & $(1/4,1/2,3/4,1)$ & $(1/5,2/5,3/5,4/5)$ \\
\midrule
4  & $\ee{4.2006}{-6}$ & $\ee{7.5819}{-9}$  & $\ee{7.0228}{-7}$  & $\ee{1.1170}{-5}$ \\
8  & $\ee{2.7000}{-7}$ & $\ee{2.2127}{-10}$ & $\ee{4.6277}{-8}$  & $\ee{7.1339}{-7}$ \\
16 & $\ee{1.6812}{-8}$ & $\ee{6.9759}{-12}$ & $\ee{2.9629}{-9}$  & $\ee{4.4329}{-8}$ \\
32 & $\ee{1.0510}{-9}$ & $\ee{2.2776}{-13}$ & $\ee{1.8665}{-10}$ & $\ee{2.7774}{-9}$ \\
\midrule
Numer. Rate & 3.9997 & 4.9368 & 3.9886 & 3.9964 \\
\bottomrule
\end{tabular}
\end{table}

\vspace{4mm}

\begin{table}[H]
\centering \footnotesize
\caption{Errors and convergence rates for $m=4$, $\alpha(u)=\frac23 (u-u(0))^4/[1+(u-u(0))^4]$, and $r=1$.}
\label{tab:m4-case3-rational-quartic}
\begin{tabular}{@{}rcccc@{}}
\toprule
$N$ & Gauss & Radau IIA & $(1/4,1/2,3/4,1)$ & $(1/5,2/5,3/5,4/5)$ \\
\midrule
4  & $\ee{4.5028}{-6}$ & $\ee{6.8395}{-9}$  & $\ee{6.5196}{-7}$  & $\ee{1.1967}{-5}$ \\
8  & $\ee{2.9141}{-7}$ & $\ee{1.9296}{-10}$ & $\ee{4.5679}{-8}$  & $\ee{7.6980}{-7}$ \\
16 & $\ee{1.8185}{-8}$ & $\ee{6.0246}{-12}$ & $\ee{2.9189}{-9}$  & $\ee{4.7941}{-8}$ \\
32 & $\ee{1.1353}{-9}$ & $\ee{1.9024}{-13}$ & $\ee{1.8411}{-10}$ & $\ee{2.9997}{-9}$ \\
\midrule
Numer. Rate & 4.0016 & 4.9850 & 3.9868 & 3.9984 \\
\bottomrule
\end{tabular}
\end{table}
\endgroup

\vskip 1mm
{\footnotesize \textbf{Acknowledgements}
The research of Xiangcheng Zheng is supported in part by the National Natural Science Foundation of China No.~12622126, the National Key R\&D Program of China No.~2023YFA1008903, the National Natural Science Foundation of China No.~12301555, and the Natural Science
Foundation of Shandong Province No.~ZR2025QB01.
}

\end{document}